\documentclass[a4paper]{article}
\usepackage{hyperref}
\usepackage{stmaryrd}
\usepackage{amsmath,bm}      
\usepackage{amsfonts}
\usepackage{amsthm}
\usepackage{amssymb}
\usepackage{mathrsfs}
\usepackage{tensor}
\usepackage{tikz}
\usetikzlibrary{matrix, arrows}
\usetikzlibrary{positioning}
\usetikzlibrary{decorations.pathreplacing}
\usepackage{adjustbox}
\theoremstyle{plain}
\newtheorem*{Example*}{\bfseries{\emph{Example}}}
\newtheorem*{Notation*}{\bfseries{\emph{Notation}}}
\newtheorem*{Theorem*}{\bfseries{\emph{Theorem}}}
\newtheorem*{Lemma*}{\bfseries{\emph{Lemma}}}
\newtheorem*{Proposition*}{\bfseries{\emph{Proposition}}}
\newtheorem*{Corollary*}{\bfseries{\emph{Corollary}}}
\newtheorem*{Remark*}{\bfseries{\emph{Remark}}}
\newtheorem*{Remarks*}{\bfseries{\emph{Remarks}}}
\newtheorem*{Def*}{\bfseries{\emph{Definition}}}
\newtheorem*{Conjecture*}{\bfseries{\emph{Conjecture}}}
\newtheorem*{sketch proof*}{Sketch proof}
\newtheorem{Theorem}{\bfseries{\emph{Theorem}}}[section]

\newtheorem{Lemma}[Theorem]{\bfseries{\emph{Lemma}}}
\newtheorem{Proposition}[Theorem]{\bfseries{\emph{Proposition}}}
\newtheorem{Corollary}[Theorem]{\bfseries{\emph{Corollary}}}
\newtheorem{Remark}[Theorem]{\bfseries{\emph{Remark}}}

\newcommand{\fil}{\text{Fil}}
\newcommand{\dR}{\text{dR}}
\newcommand{\cris}{\text{cris}}

  \title{Crystals of relative displays and Calabi-Yau threefolds}
  \author{Oliver Gregory\footnote{The author is supported by the ERC Consolidator Grant 681838 “K3CRYSTAL”.}} 
  \date{July 23, 2021}
   \begin{document}
   \maketitle
\begin{abstract}
Displays can be thought of as relative versions of Fontaine's notion of \emph{strongly divisible lattice} from integral $p$-adic Hodge theory. In favourable circumstances, the crystalline cohomology of a smooth projective $R$-scheme $X$ is endowed with a display-structure coming from complexes associated with the relative de Rham-Witt complex $W\Omega_{X/R}^{\bullet}$ of \cite{LZ04}. In this article, we use the crystal of relative displays of \cite{GL21} to prove a Grothendieck-Messing type result for the deformation theory of Calabi-Yau threefolds.
\end{abstract}
\section{Introduction}
Displays were introduced in \cite{Zin02} in order to extend the Dieudonn\'{e} module classification of $p$-divisible groups over perfect fields of characteristic $p>0$ to more general (not necessarily perfect) base rings $R$ in which $p$ is nilpotent. In \cite{LZ07} an expanded theory of ``higher'' displays was developed -- it can be thought of as a relative version of Fontaine's theory of strongly divisible lattices in integral $p$-adic Hodge theory \cite{Fon83}. It was shown in \cite{LZ07} and \cite{GL21} that the crystalline cohomology of a certain well-behaved class of smooth projective schemes over $R$ has the additional structure of a display. 

As an application of this theory, the deformation theory of $2$-displays over the small Witt ring was studied in \cite{LZ15} and used to completely describe the deformations of $F$-ordinary K3-type varieties (the $F$-ordinariness assumption was subsequently removed in \cite{Lau18}). The crucial tool used there is the crystal of relative $2$-displays associated to a K3-type variety. In \cite{GL21} it was shown that a crystal of relative displays exists more generally whenever the smooth projective scheme has a smooth versal deformation space. 

In this small note we study the deformation theory of Calabi-Yau threefolds via the crystal of relative displays. More precisely, we first recall and expand a little bit on the relevant theory of higher displays in sections \S\ref{Displays section} and \S\ref{relative displays section}. Then we quickly review the construction of display structure on crystalline cohomology via the Nygaard complexes in \S\ref{display on crystalline}, and the construction of the crystal of relative displays in \S\ref{Crystals of relative displays section}. Our philosophy is that this crystal gives a natural framework for crystalline deformation problems arising in mixed characteristic algebraic geometry. In \S\ref{Deformation theory of Calabi-Yau threefolds section} we prove the following Grothendieck-Messing type theorem:
\begin{Theorem*}
Let $S\twoheadrightarrow R$ be a nilpotent PD-thickening of artinian local $W(k)$-algebras with perfect residue field $k$, $\mathrm{char}(k)=p>0$. Let $X/R$ be a Calabi-Yau threefold which lifts to $\mathrm{Spec} \ S$. Suppose that the reduction $X\times_{\mathrm{Spec} R}\mathrm{Spec} \ k$ has a smooth versal deformation space. Then there are bijections
\begin{center}
\begin{tikzpicture}
 \node (a) [align=center] {$\text{deformation classes of } \ $ \\ $X/R\text{ over Spec }S \ $};
\node (b) [align=center] [right =0.8cm of a]{$ \ \text{ isomorphism classes of }$ \\ $\text{ geometric displays on }$ \\ $H_{\cris}^{3}(X/W(S))$};
\node (c) [align=center] [below =0.5cm of b]{$ \ \text{CY-type liftings of the}$ \\ \ $\text{ Hodge filtration of the }$ \\ $\text{relative display on }$ \\ $H_{\cris}^{3}(X/W(S))$};
\node (d) [left =0.8cm of c]{};
 \draw[<->](a) to node [above]{\emph{1:1}}(b);
 \draw[<->](c) to node [above]{\emph{1:1}}(d);
 \draw [decorate,decoration={brace,amplitude=2pt},xshift=0pt,yshift=0pt]
(-1.9,-0.5) -- (-1.9,0.5)node [black,xshift=0pt] {};
 \draw [decorate,decoration={brace,amplitude=2pt},xshift=0pt,yshift=0pt]
(1.75,0.5) -- (1.75,-0.5)node [black,xshift=0pt] {};
 \draw [decorate,decoration={brace,amplitude=2pt},xshift=0pt,yshift=0pt]
(3,-0.7) -- (3,0.7)node [black,xshift=0pt] {};
\draw [decorate,decoration={brace,amplitude=2pt},xshift=0pt,yshift=0pt]
(6.8,0.7) -- (6.8,-0.7)node [black,xshift=0pt] {};
 \draw [decorate,decoration={brace,amplitude=2pt},xshift=0pt,yshift=-50pt]
(3,-1.3) -- (3,0.5)node [black,xshift=0pt] {};
\draw [decorate,decoration={brace,amplitude=2pt},xshift=0pt,yshift=-50pt]
(6.8,0.55) -- (6.8,-1.3)node [black,xshift=0pt] {};
\end{tikzpicture}
\end{center}  
\end{Theorem*}
Let us remark that the proof of the theorem does not really use displays at all and is closer to the techniques of \cite{Del81b}. We have chosen to rephrase the result in terms of displays at the end of the article to put it in the context of \cite{LZ15}. Our theorem is a lot weaker than the main result of \cite{LZ15} for K3-type varieties - the main reason is that the deformation theory of $3$-displays is significantly more complicated than is the case for $2$-displays. It would be interesting to understand the deformation theory of the class of $3$-displays appearing in \S\ref{Deformation theory of Calabi-Yau threefolds section}.  

\ \\  
{\bf{Acknowledgements}} \ \ The results in this paper formed a part of my PhD thesis. It is a pleasure to thank my supervisor Andreas Langer for introducing me to this topic and for all of his help. Most of this work was written whilst I was visiting Universit\"{a}t Bielefeld; I thank my host Thomas Zink for his interest and hospitality. I also thank the referee for pointing out a number of mistakes and generally helping improve the paper. 
\section{Crystalline cohomology over $W(R)$}
For $p$ a prime, let $R$ be a noetherian $\mathbb{Z}_{(p)}$-algebra in which $p$ is nilpotent. For each $n\in\mathbb{N}$ let $W_{n}(R)$ be the ring of truncated Witt vectors of length $n$, so in particular $R=W_{1}(R)$ and $W(R)\cong\varprojlim_{n}W_{n}(R)$. 
\par
For any $\mathbb{Z}_{(p)}$-algebra S, the ideal $pS\subset S$ has a PD-structure $\gamma$ given on $x=py\in pS$ by
\begin{equation*}
\gamma_{m}(x):=\frac{p^{m}}{m!}y^{m} \ \ \ \ \ \ \forall m\in\mathbb{N}
\end{equation*}
(this makes sense because $\frac{p^{m-1}}{m!}\in\mathbb{Z}_{(p)}$, and is independent of the choice of $y$).  In particular, we have a PD-structure on $pW_{n}(R)\subset W_{n}(R)$ and this is compatible with the PD-structure $\gamma$ on $I_{R}:=\tensor*[^V]{W_{n-1}(R)}{}\subset W_{n}(R)$ by
\begin{equation*}
\gamma_{m}(\tensor*[^V]{\xi}{}):=\frac{p^{m-1}}{m!}\tensor*[^V]{(\xi^{m})}{} \ \ \ \ \ \ \forall m\in\mathbb{N}, \ \xi\in W_{n-1}(R)
\end{equation*}
Let $\text{Spec }W_{n}(R)$ be the PD-scheme with respect to the PD-ideal $I_{R}$. For a quasi-compact $R$-scheme $X$ we write $\text{Cris}(X/W_{n}(R))$ for the crystalline site of $X$ relative to $\text{Spec }W_{n}(R)$. The objects are PD-thickenings $(U\hookrightarrow T,\delta)$ of Zariski open subsets $U\subset X$, that is $U\hookrightarrow T$ is a closed $W_{n}(R)$-immersion defined by an ideal sheaf $\mathcal{J}$ whose PD-structure $\delta$ is compatible with $\gamma$, and such that $p$ is nilpotent on $T$. The morphisms $(U\hookrightarrow T,\delta)\rightarrow (U'\hookrightarrow T',\delta')$ are commutative diagrams
\begin{center}
\begin{tikzpicture}[descr/.style={fill=white,inner sep=1.5pt}]
        \matrix (m) [
            matrix of math nodes,
            row sep=2.5em,
            column sep=2.5em,
            text height=1.5ex, text depth=0.25ex
        ]
        { U' & T'  \\
          U & T & \\
        };

        \path[overlay,->, font=\scriptsize]
        (m-2-2) edge (m-1-2)
        (m-2-1) edge (m-1-1);
                
        \path[overlay,right hook->, font=\scriptsize]
        (m-1-1) edge (m-1-2)
        (m-2-1) edge (m-2-2);     
       
\end{tikzpicture} 
\end{center} 
where $U\rightarrow U'$ is a Zariski inclusion and $T\rightarrow T'$ is a PD-morphism. The coverings are the sets of morphisms $\{(U_{i}\hookrightarrow T_{i},\delta_{i})\rightarrow (U\hookrightarrow T,\delta)\}$ where the $T_{i}\rightarrow T$ are open immersions and $T=\cup T_{i}$.
\par
Let $(X/W_{n}(R))_{\cris}$ be the associated topos. Then for each $n\geq m$ we have a canonical morphism of topoi
\begin{equation*}
\iota_{m,n}:(X/W_{m}(R))_{\cris}\rightarrow (X/W_{n}(R))_{\cris}
\end{equation*}
Given a sheaf $\mathcal{E}_{n}$ on $\text{Cris}(X/W_{n}(R))$, let $\mathbb{R}\Gamma_{\cris}(X/W_{n}(R),\mathcal{E}_{n})$ denote the object in the derived category of $W_{n}(R)$-modules computing the cohomology of $(X/W_{n}(R))_{\cris}$ with coefficients in $\mathcal{E}_{n}$.
\par
Now suppose that $X/R$ is proper and smooth, and let $\{\mathcal{E}_{n}\}_{n\in\mathbb{N}}$ be a compatible system of locally free crystals on $\{\text{Cris}(X/W_{n}(R))\}_{n\in\mathbb{N}}$, in the sense that for all $n\geq m$ we have
\begin{equation*}
\iota_{m,n}^{\ast}\mathcal{E}_{n}\simeq\mathcal{E}_{m}
\end{equation*}
\begin{Proposition}
There is a perfect object $\mathbb{R}\Gamma_{\cris}(X/W(R),\mathcal{E})$ in the derived category of $W(R)$-modules such that 
\begin{equation*}
\mathbb{R}\Gamma_{\mathrm{cris}}(X/W(R),\mathcal{E})\otimes_{W(R)}^{\mathbb{L}}W_{n}(R)\simeq\mathbb{R}\Gamma_{\mathrm{cris}}(X/W_{n}(R),\mathcal{E}_{n})
\end{equation*}
for each $n\in\mathbb{N}$.
\end{Proposition}
\begin{proof}
By \cite[\href{https://stacks.math.columbia.edu/tag/09AR}{Tag 09AR}]{Stacks} it suffices to prove that for all $n\in\mathbb{N}$ we have
\begin{equation*}
\mathbb{R}\Gamma_{\cris}(X/W_{n+1}(R),\mathcal{E}_{n+1})\otimes_{W_{n+1}(R)}^{\mathbb{L}}W_{n}(R)\simeq\mathbb{R}\Gamma_{\cris}(X/W_{n}(R),\mathcal{E}_{n})
\end{equation*}
and that $\mathbb{R}\Gamma_{\cris}(X/R,\mathcal{E}_{1})$ is perfect. The first fact is the base-change theorem for crystalline cohomology (\cite{BO78}, 7.8), and the second fact follows from the comparison with de Rham cohomology:
\begin{equation*}
\mathbb{R}\Gamma_{\cris}(X/R,\mathcal{E}_{1})\simeq\mathbb{R}\Gamma_{\text{Zar}}(X,\mathcal{E}_{1,X}\otimes_{\mathcal{O}_{X}}\Omega_{X/R}^{\bullet}) 
\end{equation*}
This is (\cite{BO78},7.1) in the special case that $X$ is smooth. The right hand side is perfect since we assumed $R$ to be noetherian and the sheaves of differentials $\Omega_{X/R}^{n}$ are locally free of finite type. Then
\begin{equation*}
\mathbb{R}\Gamma_{\cris}(X/W(R),\mathcal{E}):=\mathbb{R}\varprojlim_{n}\mathbb{R}\Gamma_{\cris}(X/W_{n}(R),\mathcal{E}_{n})
\end{equation*} 
is the object we seek.
\end{proof}
\begin{Def*}
The crystalline cohomology of $X$ over $W(R)$ with coefficients in $\mathcal{E}$ is defined to be
\begin{equation*}
H^{m}_{\mathrm{cris}}(X/W(R),\mathcal{E}):=\mathbb{R}^{m}\Gamma_{\mathrm{cris}}(X/W(R),\mathcal{E})
\end{equation*}
If $\mathcal{E}=\{\mathcal{O}_{X/W_{n}(R)}\}_{n\in\mathbb{N}}=\mathcal{O}_{X/W(R)}$, then the crystalline cohomology of $X$ over $W(R)$ is
\begin{equation*}
H_{\mathrm{cris}}^{m}(X/W(R)):=H^{m}_{\mathrm{cris}}(X/W(R),\mathcal{O}_{X/W(R)})
\end{equation*}
\end{Def*}
Let $S$ be a $\mathbb{Z}_{(p)}$-scheme and $X$ an $S$-scheme. In \cite{LZ04}, the relative de Rham-Witt complex $W\Omega_{X/S}^{\bullet}=\{W_{n}\Omega_{X/S}^{\bullet}\}_{n\in\mathbb{N}}$ is constructed as the initial object in the category of $F$-$V$-procomplexes. In the case that $S=\text{Spec }k$ is the spectrum of a perfect field of characteristic $p$ the relative de Rham-Witt complex recovers the de Rham-Witt complex of Bloch-Illusie. In our situation, the hypercohomology of $W\Omega_{X/R}^{\bullet}$ computes the crystalline cohomology of proper smooth schemes $X/R$:-
\begin{Theorem}\label{comparison}
Let $R$ be a noetherian $\mathbb{Z}_{(p)}$-algebra in which $p$ is nilpotent, and let $X/\text{Spec R}$ be a proper smooth scheme. Then there is a canonical isomorphism
\begin{equation*}
H_{\mathrm{cris}}^{m}(X/W(R))\cong\mathbb{H}^{m}(X,W\Omega_{X/R}^{\bullet})
\end{equation*}
\end{Theorem}
\begin{proof}
This is (\cite{LZ04}, Theorem 3.5).
\end{proof}

\section{Displays}\label{Displays section}
Fix a prime number $p$. In this preliminary section we shall briefly collect together some aspects of the theory of displays, mainly drawing from \cite{LZ07}, \cite{LZ15} and \cite{Wid12}.
\begin{Def*}
A \emph{frame} $\mathscr{F}$ is a quintuple $(W,J,R,\sigma,\dot{\sigma})$, where $R$ and $W$ are rings and $J\subseteq W$ is an ideal such that $W/J=R$. The map $\sigma:W\rightarrow W$ is an endomorphism and $\dot{\sigma}:J\rightarrow W$ is a $\sigma$-linear $W$-module homomorphism. We require the following conditions to be satisfied:-
\par
(i) $J+pW\subseteq\mathcal{R}ad(W)$ 
\par
(ii) $\sigma(a)\equiv a^{p}\mod pW$ for all $a\in W$
\par
(iii) $\dot{\sigma}(J)$ generates $W$ as a $W$-module.
\end{Def*}

By definition, condition (iii) says that the linearisation $\dot{\sigma}^{\#}:J^{(\sigma)}\rightarrow W$ is surjective. Choose $b\in J^{(\sigma)}$ such that $\dot{\sigma}^{\#}(b)=1$ and set $\theta=\sigma^{\#}(b)$. Then for any $a\in J$, we compute
\begin{align*}
\sigma(a)
& =\dot{\sigma}^{\#}(b)\sigma(a)\\
& =\dot{\sigma}^{\#}(ab)\\
& =\sigma^{\#}(b)\dot{\sigma}(a)\\
& =\theta\dot{\sigma}(a)
\end{align*}
so we see that there is a unique $\theta\in W$ such that $\sigma(a)=\theta\dot{\sigma}(a)$ for every $a\in J$. We shall always assume that $\theta=p$.

\begin{Def*}
A \emph{Verj\"{u}ngung} for a frame $\mathscr{F}=(W,J,R,\sigma,\dot{\sigma})$ is the data of two $W$-module homomorphisms
\begin{equation*}
\nu:J\otimes_{W}J\rightarrow J \ \ \text{ and } \ \ \pi:J\rightarrow J
\end{equation*} 
satisfying the following four conditions for all $\eta,\eta'\in J$:-
\begin{align*}
& \text{(i)} \ \ \hspace{0.5cm}\pi(\nu(\eta,\eta'))=\eta\eta'\\
& \text{(ii)} \ \hspace{0.5cm}\dot{\sigma}(\nu(\eta,\eta'))=\dot{\sigma}(\eta)\dot{\sigma}(\eta')\\
& \text{(iii)} \hspace{0.5cm}\dot{\sigma}(\pi(\eta))=\sigma(\eta)\\
& \text{(iv)} \hspace{0.5cm}\ker\dot{\sigma}\cap\ker\pi=0
\end{align*}
\end{Def*}

The iterations $\nu^{(k)}:J\otimes_{W}\cdots\otimes_{W}J\rightarrow J$ of the Verj\"{u}ngung are well-defined, and we write $J_{k}$ for the image of $\nu^{(k)}$. 

Before defining what a display over a frame with Verj\"{u}ngung $\mathscr{F}$ is, we shall first introduce the intermediate concept of a predisplay; the category of $\mathscr{F}$-displays will then be a certain well-behaved full subcategory of the category of $\mathscr{F}$-predisplays.
\begin{Notation*}
In the above situation, given a $\sigma$-linear map of $W$-modules $f:M\rightarrow N$, we will denote by $\tilde{f}:J\otimes_{W}M\rightarrow N$ the $\sigma$-linear map given by $\tilde{f}(\eta\otimes m)=\dot{\sigma}(\eta)f(m)$. 
\end{Notation*}
\begin{Def*}
Let $\mathscr{F}=(W,J,R,\sigma,\dot{\sigma})$ be a frame. An \emph{$\mathscr{F}$-predisplay} $\mathcal{P}=(P_{i},\iota_{i},\alpha_{i},F_{i})$ is the following data: for each $i\geq 0$ we have
\newline
\par
(i) a $W$-module $P_{i}$
\par
(ii) two $W$-module homomorphisms 
\begin{equation*}
\iota_{i}:P_{i+1}\rightarrow P_{i} \ \ \text{ and 
} \ \ \alpha_{i}:J\otimes_{W}P_{i}\rightarrow P_{i+1}  
\end{equation*}
\par
(iii) a $\sigma$-linear map $F_{i}:P_{i}\rightarrow P_{0}$.
\newline
\newline
The data is required to satisfy the following properties:-
\newline
\par 
(P1) $\iota_{i}\circ\alpha_{i}=\alpha_{i-1}\circ(\text{id}_{J}\otimes\iota_{i-1})$ for $i\geq 1$
\par 
(P2) $\iota_{i}\circ\alpha_{i}:J\otimes P_{i}\rightarrow P_{i}$ is the multiplication map for each $i\geq 0$
\par 
(P3) $F_{i+1}\circ\alpha_{i}=\tilde{F}_{i}$ for all $i\geq 0$.
\end{Def*}
An immediate consequence of the definition is that $F_{i}\circ\iota_{i}=pF_{i+1}$. Indeed, let $\eta\in J$ and $x\in P_{i+1}$. Then we compute
\begin{equation*}
p\tilde{F}_{i}(\eta\otimes\iota_{i}(x))=pF_{i+1}\circ\alpha_{i}(\eta\otimes\iota_{i}(x))=F_{i+1}(\eta x)=p\sigma(\eta)F_{i+1}(x)
\end{equation*}
but 
\begin{equation*}
p\tilde{F}_{i}(\eta\otimes\iota_{i}(x))=p\dot{\sigma}(\eta)F_{i}(\iota_{i}(x))=\sigma(\eta)F_{i}(\iota_{i}(x))
\end{equation*}
The idea is that we should think of the $F_{i}$ as divided Frobenius maps and we shall often refer to them as such.
\begin{Def*}
Let $\mathcal{P}=(P_{i},\iota_{i},\alpha_{i},F_{i})$ and $\mathcal{P}'=(P'_{i},\iota'_{i},\alpha'_{i},F'_{i})$ be two $\mathscr{F}$-predisplays. A \emph{morphism of predisplays} $\psi:\mathcal{P}\rightarrow\mathcal{P}'$ is a sequence of $W$-module homomorphisms $\psi_{i}:P_{i}\rightarrow P'_{i}$ such that the following squares commute for each $i\geq 0$:-
\begin{center}
\begin{tikzpicture}[descr/.style={fill=white,inner sep=1.5pt}]
        \matrix (m) [
            matrix of math nodes,
            row sep=2.5em,
            column sep=2.5em,
            text height=1.5ex, text depth=0.25ex
        ]
        { P_{i+1} & P_{i} & J\otimes_{W}P_{i} & P_{i+1} & P_{i} & P_{0} \\
          P'_{i+1} & P'_{i} & J\otimes_{W}P'_{i} & P'_{i+1} & P'_{i} & P'_{0}\\
        };

        \path[overlay,->, font=\scriptsize]
        (m-1-1) edge node [left]{$\psi_{i+1}$} (m-2-1)
        (m-1-2) edge node [right]{$\psi_{i}$} (m-2-2)
        (m-1-3) edge node [left]{$\text{id}_{J}\otimes\psi_{i}$} (m-2-3)
        (m-1-4) edge node [right]{$\psi_{i+1}$} (m-2-4)
        (m-1-5) edge node [left]{$\psi_{i}$} (m-2-5)
        (m-1-6) edge node [right]{$\psi_{0}$} (m-2-6)
        (m-1-1) edge node [above]{$\iota_{i}$} (m-1-2)
        (m-1-3) edge node [above]{$\alpha_{i}$} (m-1-4)
        (m-1-5) edge node [above]{$F_{i}$} (m-1-6)
        (m-2-1) edge node [above]{$\iota'_{i}$} (m-2-2)
        (m-2-3) edge node [above]{$\alpha'_{i}$} (m-2-4)
        (m-2-5) edge node [above]{$F'_{i}$} (m-2-6);
\end{tikzpicture}
\end{center} 
\end{Def*}
In this way we have the category $\mathscr{F}$-$\underline{\mathcal{P}redisp}$ of $\mathscr{F}$-predisplays.
\par
Let $\mathcal{P}$ and $\mathcal{P}'$ be $\mathscr{F}$-predisplays as in the definition above. Then we define the direct sum of $\mathcal{P}$ and $\mathcal{P'}$ as 
\begin{equation*}
\mathcal{P}\oplus\mathcal{P'}:=(P_{i}\oplus P'_{i},\iota_{i}\oplus\iota'_{i},\alpha_{i}\oplus\alpha'_{i},F_{i}\oplus F'_{i})
\end{equation*}
Clearly this gives another $\mathscr{F}$-predisplay.
\par
Given a morphism of predisplays $\psi:\mathcal{P}\rightarrow\mathcal{P}'$, we can define the \emph{kernel} of $\psi$ as
\begin{equation*}
\mathcal{K}er\ \psi:=(\ker\psi_{i},\bar{\iota}_{i},\bar{\alpha}_{i},\bar{F}_{i})
\end{equation*}
where $\bar{\iota}_{i}=\iota_{i}|_{\ker\psi_{i+1}}$, $\bar{\alpha}_{i}=\alpha_{i}|_{J\otimes_{W}\ker\psi_{i}}$ and $\bar{F}_{i}=F_{i}|_{\ker\psi_{i}}$. Similarly we can also define $\mathcal{C}oker\ \psi$. It is immediate that $\mathcal{K}er\ \psi$ and $\mathcal{C}oker\ \psi$ are both $\mathscr{F}$-predisplays.
\begin{Corollary}
$\mathscr{F}$-$\underline{\mathcal{P}redisp}$ is an abelian category.
\end{Corollary} 
When $\mathscr{F}$ is a frame with Verj\"{u}ngung, we can build $\mathscr{F}$-predisplays using \emph{standard data}:-
\begin{Def*}
Let $\mathscr{F}=(W,J,R,\sigma,\dot{\sigma},\nu,\pi)$ be a frame with Verj\"{u}ngung. A \emph{standard datum} for $\mathscr{F}$ is a a set of finitely generated projective $W$-modules $L_{0},\ldots, L_{d}$, and $\sigma$-linear homomorphisms 
\begin{equation*}
\Phi_{i}:L_{i}\rightarrow L_{0}\oplus\cdots\oplus L_{d}
\end{equation*}
for each $i=0,\ldots, d$, such that the induced map $\Phi_{0}\oplus\cdots\oplus \Phi_{d}$ is a $\sigma$-linear isomorphism.
\end{Def*}
Given a standard datum $(L_{i},\Phi_{i})$, we set
\begin{equation*}
P_{i}:=J_{i}L_{0}\oplus\cdots\oplus JL_{i-1}\oplus\cdots\oplus L_{d}
\end{equation*}
Then the maps $\iota_{i}$, $\alpha_{i}$ and $F_{i}$ are defined by the following diagrams:-
\begin{center}
\begin{tikzpicture}[descr/.style={fill=white,inner sep=1.5pt}]
        \matrix (m) [
            matrix of math nodes,
            row sep=2.5em,
            column sep=-0.3em,
            text height=1.5ex, text depth=0.25ex
        ]
        { P_{i+1} & = & J_{i+1}L_{0} & \oplus & J_{i}L_{1} & \oplus & \cdots & \oplus & JL_{i} & \oplus & L_{i+1} & \oplus & \cdots & \oplus & L_{d} \\
            P_{i} & = & J_{i}L_{0} & \oplus & J_{i-1}L_{1} & \oplus & \cdots & \oplus & L_{i} & \oplus & L_{i+1} & \oplus & \cdots & \oplus & L_{d}\\
        };

        \path[overlay,->, font=\scriptsize]
        (m-1-1) edge node [left]{$\iota_{i}$} (m-2-1)
        (m-1-3) edge node [left]{$\pi$} (m-2-3)
        (m-1-5) edge node [left]{$\pi$} (m-2-5)
        (m-1-9) edge node [left]{$\text{id}$} (m-2-9)
        (m-1-11) edge node [left]{$\text{id}$} (m-2-11)
        (m-1-15) edge node [left]{$\text{id}$} (m-2-15);
\end{tikzpicture}
\end{center}
\begin{center}
\begin{adjustbox}{width=12cm}
\begin{tikzpicture}[descr/.style={fill=white,inner sep=1.5pt}]
        \matrix (m) [
            matrix of math nodes,
            row sep=2.5em,
            column sep=-0.3em,
            text height=1.5ex, text depth=0.25ex
        ]
        { J\otimes P_{i} & = & J\otimes J_{i}L_{0} & \oplus & J\otimes J_{i-1}L_{1} & \oplus & \cdots & \oplus & J\otimes L_{i} & \oplus & J\otimes L_{i+1} & \oplus & \cdots & \oplus & J\otimes L_{d} \\
            P_{i+1} & = & J_{i+1}L_{0} & \oplus & J_{i}L_{1} & \oplus & \cdots & \oplus & JL_{i} & \oplus & L_{i+1} & \oplus & \cdots & \oplus & L_{d}\\
        };

        \path[overlay,->, font=\scriptsize]
        (m-1-1) edge node [left]{$\alpha_{i}$} (m-2-1)
        (m-1-3) edge node [left]{$\nu$} (m-2-3)
        (m-1-5) edge node [left]{$\nu$} (m-2-5)
        (m-1-9) edge node [left]{$\nu_{1}=\text{id}$} (m-2-9)
        (m-1-11) edge node [left]{$\text{mult}$} (m-2-11)
        (m-1-15) edge node [left]{$\text{mult}$} (m-2-15);
\end{tikzpicture}
\end{adjustbox}
\end{center}
\begin{center}
\begin{tikzpicture}[descr/.style={fill=white,inner sep=1.5pt}]
        \matrix (m) [
            matrix of math nodes,
            row sep=2.5em,
            column sep= 0em,
            text height=1.5ex, text depth=0.25ex
        ]
        { P_{i} & = & J_{i}L_{0} & \oplus & \cdots & \oplus & JL_{i-1} & \oplus & L_{i} & \oplus & L_{i+1} & \oplus & L_{i+2} & \oplus & \cdots & \oplus & L_{d}\\
            P_{0} & = & L_{0} & \oplus & \cdots & \oplus & L_{i-1} & \oplus & L_{i} & \oplus & L_{i+1} & \oplus & L_{i+2} & \oplus & \cdots & \oplus & L_{d} \\
        };

        \path[overlay,->, font=\scriptsize]
        (m-1-1) edge node [left]{$F_{i}$} (m-2-1)
        (m-1-3) edge node [left]{$\tilde{\Phi}_{0}$} (m-2-3)
        (m-1-7) edge node [left]{$\tilde{\Phi}_{i-1}$} (m-2-7)
        (m-1-9) edge node [left]{$\Phi_{i}$} (m-2-9)
        (m-1-11) edge node [left]{$p\Phi_{i+1}$} (m-2-11)
        (m-1-13) edge node [left]{$p^{2}\Phi_{i+2}$} (m-2-13)
        (m-1-17) edge node [left]{$p^{d-i}\Phi_{d}$} (m-2-17);
       
\end{tikzpicture}
\end{center}
\begin{Lemma}
The tuple $(P_{i}, \iota_{i},\alpha_{i},F_{i})$ constructed from the standard datum $(L_{i},\Phi_{i})$ in the above fashion is an $\mathscr{F}$-predisplay.
\end{Lemma}
\begin{proof}
We verify the properties (P1),(P2) and (P3) in turn. The strategy is to look at the individual direct summands:-
\par
(P1) This is immediate from the definition of Verj\"{u}ngung.
\par
(P2) We are required to prove that $\iota_{i}\circ\alpha_{i}=\text{mult}$. This is obvious for the $j^{th}$-summand with $j>i$, and also true for $j=i$ by the definition of Verj\"{u}ngung. For $j<i$, we need to check that $\pi\nu=\text{mult}$. So choose $x\in L_{j}$ and $\eta\in J$, $\eta'\in J_{i-j}$. Then the map $\iota_{i}\circ\alpha_{i}$ on the $j^{th}$-summand is
\begin{align*}
& \ \ J\otimes_{W}J_{i-j}L_{j}\xrightarrow{\nu}J_{i-j+1}L_{j}\xrightarrow{\pi}J_{i-j}L_{j} \\
& \ \eta\otimes\eta'\otimes x \ \mapsto \ \nu(\eta,\eta')\otimes x \ \mapsto \ \pi(\nu(\eta,\eta'))\otimes x
\end{align*}
but $\pi(\nu(\eta,\eta'))=\eta\eta'$.
\par
(P3) We are required to show $F_{i+1}\circ\alpha_{i}=\tilde{F}_{i}$ for $i\geq 0$. Consider the $j^{th}$-summand. Then we have three cases:- 
\par
The case $j=i$ is the trivial statement $\tilde{\Phi}_{j}=\tilde{\Phi}_{j}$. If $j>i$ then the map $\tilde{F}_{i}$ on the $j^{th}$-summand is $p^{j-i}\tilde{\Phi}_{j-i}$, $F_{i+1}$ is $p^{j-i-1}\Phi_{j+1}$ and $\alpha_{i}$ is $\text{mult}$. Let $x\in L_{j}$ and $\eta\in J$. Then we see that
\begin{align*}
F_{i+1}\circ\alpha_{i}(\eta\otimes x)
& = F_{i+1}(\eta x) \\
& = p^{j-i-1}\Phi_{j-i}(\eta x) \\
& = p^{j-i-1}\sigma(\eta)\Phi_{j-i}(x) \\
& = p^{j-i}\dot{\sigma}(\eta)\Phi_{j-i}(x) \\
& = p^{j-1}\tilde{\Phi}_{j-i}(\eta\otimes x) \\
& = \tilde{F}_{i}(\eta\otimes x)
\end{align*}
The final case is when $j<i$. In this case, the map $\tilde{F}_{i}$ on the $j^{th}$-summand is $\tilde{\tilde{\Phi}}_{j}$, $F_{i+1}$ is $\tilde\Phi_{j}$ and $\alpha_{i}$ is $\nu$. Let $x\in L_{j}$, $\eta\in J$ and $\eta'\in J_{i-j}$. Then we see that
\begin{align*}
F_{i+1}\circ\alpha_{i}(\eta\otimes\eta'\otimes x)
& =\tilde{\Phi}_{j}(\nu(\eta,\eta')\otimes x) \\
& =\dot{\sigma}(\nu(\eta,\eta'))\Phi_{j}(x) \\
& =\dot{\sigma}(\eta)\dot{\sigma}(\eta')\Phi_{j}(x) \\
& =\dot{\sigma}(\eta)\tilde{\Phi}_{j}(\eta'\otimes x) \\
& =\tilde{\tilde{\Phi}}_{j}(\eta\otimes\eta'\otimes x) \\
& =\tilde{F}_{i}(\eta\otimes\eta'\otimes x)
\end{align*}
\end{proof}
\begin{Def*}
The $\mathscr{F}$-predisplay constructed from the standard datum $(L_{i},\Phi_{i})$ $i=0,\ldots,d$ as above is called the \emph{$d$-$\mathscr{F}$-display} of the standard datum $(L_{i},\Phi_{i})$. A \emph{$d$-$\mathscr{F}$-display} is an $\mathscr{F}$-predisplay which is isomorphic to the $d$-$\mathscr{F}$-display of a standard datum. Such an isomorphism is called a \emph{normal decomposition}. 
\end{Def*}
A morphism of displays is simply a morphism of predisplays, and we denote the full subcategory of $\mathscr{F}$-displays by $\mathscr{F}$-$\underline{\mathcal{D}isp}$. One can define the usual linear algebra operations such as tensor product, duals, base change and so forth for $\mathscr{F}$-displays by appealing to the normal decomposition. These definitions turn out to be independent (up to canonical isomorphism) of the choice of normal decomposition. This gives $\mathscr{F}$-$\underline{\mathcal{D}isp}$ the structure of an exact tensor category. Note that it is not an abelian category, since ``quotients'' of displays will fail to be displays in general.

\begin{Def*}
Let $\mathscr{F}=(W,J,R,\sigma,\dot{\sigma},\nu,\pi)$ and $\mathscr{F}'=(W',J',R',\sigma',\dot{\sigma}',\nu',\pi')$ be two frames with Verj\"{u}ngung, and suppose that $u:\mathscr{F}'\rightarrow\mathscr{F}$ is a frame homomorphism. Let $\mathscr{D}=(P_{i},\iota_{i},\alpha_{i}, F_{i})$ be an $\mathscr{F}$-predisplay. Then the \emph{pullback $u^{\bullet}\mathscr{D}=(P'_{i},\iota'_{i},\alpha'_{i}, F'_{i})$ of $\mathscr{D}$ along $u$} is the $\mathscr{F}'$-predisplay where $P'_{i}=P_{i}$ but considered as $W'$-modules, $\iota'_{i}=\iota_{i}$, $F'_{i}=F_{i}$ and $\alpha'_{i}$ is the composite 
\begin{equation*}
J'\otimes_{W'}P'_{i}\rightarrow J\otimes_{W}P'_{i}\xrightarrow{\alpha_{i}}P'_{i+1}
\end{equation*}
The functor $u^{\bullet}:\mathscr{F}$-$\underline{\mathcal{D}isp}\rightarrow\mathscr{F'}$-$\underline{\mathcal{D}isp}$ has a left adjoint 
\begin{equation*}
u_{\bullet}:\mathscr{F'}\text{-}\underline{\mathcal{D}isp}\rightarrow\mathscr{F}\text{-}\underline{\mathcal{D}isp}
\end{equation*}
(see \cite{LZ15}, Prop. 6). If $\mathscr{D}'$ is an $\mathscr{F}'$-display, then the $\mathscr{F}$-display $u_{\bullet}\mathscr{D}'$ is called the \emph{base-change of $\mathscr{D}'$ along $u$}. If $(L'_{i},\Phi'_{i})$ is a standard datum for $\mathscr{D}'$, then $(W\otimes_{W'}L'_{i},\sigma\otimes\Phi_{i})$ is a standard datum for $u_{\bullet}\mathscr{D}'$. 
\end{Def*}
\section{$\mathscr{W}_{S/R}$-displays}\label{relative displays section}
Let $S\twoheadrightarrow R$ be a surjection of $p$-adic rings, with kernel $\mathfrak{a}$ equipped with a PD-structure, and such that $\mathfrak{a}$ is nilpotent modulo $pS$. In this note we are interested in displays over the relative Witt frame $\mathscr{W}_{S/R}$ of $S\twoheadrightarrow R$, which we now describe:

Let $I_{S}=\tensor*[^V]{W(S)}{}$ denote the image of Verschiebung. The PD-structure on $\mathfrak{a}$ allows us to define the divided Witt polynomials ${\bf{w}}_{n}/p^{n}$ on the ideal $W(\mathfrak{a})\subset W(S)$. The divided Witt polynomials give an isomorphism of additive groups
\begin{equation*}
\log:W(\mathfrak{a})\xrightarrow{\sim}\prod_{\mathbb{N}}\mathfrak{a}
\end{equation*}
Then $\tilde{a}:=\log^{-1}(\mathfrak{a})=(\mathfrak{a},0,0,\ldots)$ is an ideal in $W(S)$. In this way, we view $\mathfrak{a}$ as an ideal in $W(S)$.

Writing $\mathcal{J}=\ker(W(S)\rightarrow S\rightarrow R)$, we have a direct sum decomposition $\mathcal{J}=\tilde{a}\oplus I_{S}$, with $\sigma(\tilde{a})=0$ and $\tilde{a}\cdot I_{S}=0$, where $\sigma:W(S)\rightarrow W(S)$ is the Witt vector Frobenius. We define $\dot{\sigma}:\mathcal{J}\rightarrow W(S)$ using this decomposition; $\dot{\sigma}(a+\tensor*[^V]{\xi}{})=\xi$. Since $S$ is $p$-adically complete, $W(S)$ is both $p$-adically and $I_{S}$-adically complete (\cite{Zin02}, Prop. 3), and it follows that $\mathcal{J}+pW(S)\subseteq\mathcal{R}ad(W(S))$.

\begin{Def*}
The frame $\mathscr{W}_{S/R}=(W(S),\mathcal{J},R,\sigma,\dot{\sigma})$ is called the \emph{relative Witt frame} of $S/R$. If $S=R$, we simply write $\mathscr{W}_{R}$. 
\end{Def*}
The relative Witt frame $\mathscr{W}_{S/R}$ is endowed with a Verj\"{u}ngung $(\nu,\pi)$ as follows:
\begin{align*}
& \nu(a_{1}+\tensor*[^V]{\xi}{_1},a_{2}+\tensor*[^V]{\xi}{_2})=a_{1}a_{2}+\tensor*[^V]{(}{}\xi_{1}\xi_{2}) \\
& \pi(a+\tensor*[^V]{\xi}{})=a+p\tensor*[^V]{\xi}{}
\end{align*}

\begin{Def*}
Let $\mathscr{D}=(P_{i},\iota_{i},\alpha_{i}, F_{i})$ be a $\mathscr{W}_{S/R}$-display. The \emph{Hodge filtration} of $\mathscr{D}$ is the decreasing filtration of $P_{0}/\mathcal{J}P_{0}$ by the (direct summand) $R$-modules 
\begin{equation*}
\text{Fil}^{i}\mathscr{D}:=\text{im}\left(P_{i}\xrightarrow{\iota_{i}^{\text{iter}}}P_{0}\rightarrow P_{0}/\mathcal{J}P_{0}\right)
\end{equation*}
where $\iota_{i}^{\text{iter}}:=\iota_{i-1}\circ\cdots\circ\iota_{0}$.
\end{Def*}

\begin{Def*}
Let $\mathscr{D}=(P_{i},\iota_{i},\alpha_{i},F_{i})$ be a $\mathscr{W}_{S/R}$-display. Then a filtration by direct summands $E^{\bullet}\subseteq P_{0}/I_{S}P_{0}$ lifting the Hodge filtration $\text{Fil}^{\bullet}\mathscr{D}$ is called \emph{admissible} if 
\begin{equation*}
E^{i}\subseteq\text{im}\left(P_{i}\xrightarrow{\iota_{i}^{\text{iter}}}P_{0}\rightarrow P_{0}/I_{S}P_{0}\right)
\end{equation*}
for each $i\geq 0$.
\end{Def*}
\begin{Proposition}
Let $u:\mathscr{W}_{S}\rightarrow\mathscr{W}_{S/R}$. Then there is an equivalence of categories
\begin{center}
\begin{tikzpicture}
 \node (a) [align=center] {$\mathcal{W}_{S}\text{-displays } \ $};
\node (b) [align=center] [right =0.8cm of a]{$\mathcal{W}_{S/R}\text{-displays together with an}$ \\ $ \ \text{ admissible lifting of the Hodge filtration}$};
\node (z) [below =0.4cm of a]{};
\node (c) [right =0.1cm of z]{$\mathscr{D}$};
\node (d) [right =2.9cm of c]{$(u_{\bullet}\mathscr{D},\text{Fil}^{\bullet}\mathscr{D}$)};
 \draw[->](a) to node [above]{$\sim$}(b);
 \draw[|->](c) to node {}(d);
 \draw [decorate,decoration={brace,amplitude=2pt},xshift=0pt,yshift=0pt]
(-1.1,-0.25) -- (-1.1,0.25)node [black,xshift=0pt] {};
 \draw [decorate,decoration={brace,amplitude=2pt},xshift=0pt,yshift=0pt]
(1,0.25) -- (1,-0.25)node [black,xshift=0pt] {};
 \draw [decorate,decoration={brace,amplitude=2pt},xshift=0pt,yshift=0pt]
(2.2,-0.5) -- (2.2,0.5)node [black,xshift=0pt] {};
\draw [decorate,decoration={brace,amplitude=2pt},xshift=0pt,yshift=0pt]
(8.6,0.5) -- (8.6,-0.5)node [black,xshift=0pt] {};
\end{tikzpicture}
\end{center} 
Given a $\mathscr{W}_{S/R}$-display $\mathscr{D}$ and an admissible lifting $E^{\bullet}$ of $\text{Fil}^{\bullet}\mathscr{D}$, we shall write $\mathscr{D}_{E^{\bullet}}$ for the $\mathscr{W}_{S}$-display obtained from the above equivalence. By construction we have $\text{Fil}^{\bullet}\mathscr{D}_{E^{\bullet}}=E^{\bullet}$.
\end{Proposition}
\begin{proof}
This is (\cite{LZ15}, Corollary 11).
\end{proof}

\section{Display-structure on crystalline cohomology}\label{display on crystalline}
For smooth projective schemes whose cohomology behaves well with respect to base-change, \cite{LZ07} showed that the Nygaard complexes endow crystalline cohomology, which \emph{a priori} are $F$-crystals, with the richer structure of a display. We briefly recall the construction:
\par
Let $X/\text{Spec }R$ be a smooth projective scheme. Let $\mathcal{N}^{i}W_{n}\Omega_{X/R}^{\bullet}$ denote the complex of $W_{n}(R)$-modules
\begin{equation*}
W_{n-1}\mathcal{O}_{X}\xrightarrow{d}\cdots\xrightarrow{d}W_{n-1}\Omega_{X/R}^{i-1}\xrightarrow{dV}W_{n}\Omega_{X/R}^{i}\xrightarrow{d}W_{n}\Omega_{X/R}^{i}\xrightarrow{d}\cdots
\end{equation*} 
where we consider $W_{n-1}\Omega_{X/R}^{j}$ as a $W_{n}(R)$-module via restriction of scalars along the Witt vector Frobenius $\tensor*[^F]{}{}:W_{n}(R)\rightarrow W_{n-1}(R)$. Then $\mathcal{N}^{0}W_{n}\Omega_{X/R}^{\bullet}=W_{n}\Omega_{X/R}^{\bullet}$. Write $I_{R,n}:=\tensor*[^V]{W_{n-1}(R)}{}$. Then there are morphisms of complexes 
\begin{align*}
& \hat{\iota}_{i}:\mathcal{N}^{i+1}W_{n}\Omega_{X/R}^{\bullet}\rightarrow\mathcal{N}^{i}W_{n}\Omega_{X/R}^{\bullet} \\
& \hat{F}_{i}:\mathcal{N}^{i}W_{n}\Omega_{X/R}^{\bullet}\rightarrow W_{n-1}\Omega_{X/R}^{\bullet}\\
& \hat{\alpha}_{i}:I_{R,n}\otimes_{W_{n}(R)}\mathcal{N}^{i}W_{n}\Omega_{X/R}^{\bullet}\rightarrow\mathcal{N}^{i+1}W_{n}\Omega_{X/R}^{\bullet}
\end{align*}
given by
\begin{center}
\begin{adjustbox}{width=12cm}
\begin{tikzpicture}[descr/.style={fill=white,inner sep=1.5pt}]
        \matrix (m) [
            matrix of math nodes,
            row sep=2.5em,
            column sep=-0.5em,
            text height=1.5ex, text depth=0.25ex
        ]
        {  W_{n-1}\mathcal{O}_{X} & \xrightarrow{d} & \cdots & \xrightarrow{d} & W_{n-1}\Omega_{X/R}^{i-1} & \xrightarrow{d} & W_{n-1}\Omega_{X/R}^{i} & \xrightarrow{dV} & W_{n}\Omega_{X/R}^{i+1} & \xrightarrow{d} & W_{n}\Omega_{X/R}^{i+2} & \xrightarrow{d} & \cdots \\
            W_{n-1}\mathcal{O}_{X} & \xrightarrow{d} & \cdots & \xrightarrow{d} & W_{n-1}\Omega_{X/R}^{i-1} & \xrightarrow{dV} & W_{n}\Omega_{X/R}^{i} & \xrightarrow{d} & W_{n}\Omega_{X/R}^{i+1} & \xrightarrow{d} & W_{n}\Omega_{X/R}^{i+2} & \xrightarrow{d} & \cdots\\
        };

        \path[overlay,->, font=\scriptsize]
        (m-1-1) edge node [left]{$p$} (m-2-1)
        (m-1-5) edge node [left]{$p$} (m-2-5)
        (m-1-7) edge node [left]{$V$} (m-2-7)
        (m-1-9) edge node [left]{$\text{id}$} (m-2-9)
        (m-1-11) edge node [left]{$\text{id}$} (m-2-11);
\end{tikzpicture}
\end{adjustbox}
\end{center}
\begin{center}
\begin{adjustbox}{width=12cm}
\begin{tikzpicture}[descr/.style={fill=white,inner sep=1.5pt}]
        \matrix (m) [
            matrix of math nodes,
            row sep=2.5em,
            column sep=-0.5em,
            text height=1.5ex, text depth=0.25ex
        ]
        {   W_{n-1}\mathcal{O}_{X} & \xrightarrow{d} & \cdots & \xrightarrow{d} & W_{n-1}\Omega_{X/R}^{i-1} & \xrightarrow{dV} & W_{n}\Omega_{X/R}^{i} & \xrightarrow{d} & W_{n}\Omega_{X/R}^{i+1} & \xrightarrow{d} & W_{n}\Omega_{X/R}^{i+2} & \xrightarrow{d} & \cdots\\
            W_{n-1}\mathcal{O}_{X} & \xrightarrow{d} & \cdots & \xrightarrow{d} & W_{n-1}\Omega_{X/R}^{i-1} & \xrightarrow{d} & W_{n-1}\Omega_{X/R}^{i} & \xrightarrow{d} & W_{n-1}\Omega_{X/R}^{i+1} & \xrightarrow{d} & W_{n-1}\Omega_{X/R}^{i+2} & \xrightarrow{d} & \cdots\\
        };

        \path[overlay,->, font=\scriptsize]
        (m-1-1) edge node [left]{$\text{id}$} (m-2-1)
        (m-1-5) edge node [left]{$\text{id}$} (m-2-5)
        (m-1-7) edge node [left]{$F$} (m-2-7)
        (m-1-9) edge node [left]{$pF$} (m-2-9)
        (m-1-11) edge node [left]{$p^{2}F$} (m-2-11);
\end{tikzpicture}
\end{adjustbox}
\end{center}
\begin{center}
\begin{adjustbox}{width=12cm}
\begin{tikzpicture}[descr/.style={fill=white,inner sep=1.5pt}]
        \matrix (m) [
            matrix of math nodes,
            row sep=2.5em,
            column sep=-0.5em,
            text height=1.5ex, text depth=0.25ex
        ]
        {   I_{R}\otimes W\mathcal{O}_{X} & \rightarrow & \cdots & \rightarrow & I_{R}\otimes W\Omega_{X/R}^{i-1} & \rightarrow & I_{R}\otimes W\Omega_{X/R}^{i} & \rightarrow & I_{R}\otimes W\Omega_{X/R}^{i+1} & \rightarrow & I_{R}\otimes W\Omega_{X/R}^{i+2} & \rightarrow & \cdots\\
           W\mathcal{O}_{X} & \xrightarrow{d} & \cdots & \xrightarrow{d} & W\Omega_{X/R}^{i-1} & \xrightarrow{d} & W\Omega_{X/R}^{i} & \xrightarrow{dV} & W\Omega_{X/R}^{i+1} & \xrightarrow{d} & W\Omega_{X/R}^{i+2} & \xrightarrow{d} & \cdots \\
        };

        \path[overlay,->, font=\scriptsize]
        (m-1-1) edge node [left]{} (m-2-1)
        (m-1-5) edge node [left]{} (m-2-5)
        (m-1-7) edge node [left]{$\tilde{F}$} (m-2-7)
        (m-1-9) edge node [left]{$\text{mult.}$} (m-2-9)
        (m-1-11) edge node [left]{$\text{mult.}$} (m-2-11);
\end{tikzpicture}
\end{adjustbox}
\end{center}
respectively.
We omitted the subscript in the final diagram for space, and in that diagram the first unlabelled arrows are given by
\begin{align*}
& I_{R,n}\otimes W_{n-1}\Omega_{X/R}^{j}\rightarrow W_{n-1}\Omega_{X/R}^{j} \\
& \ \ \ \ \ \ \ \ \ \ \ \ \ \ \tensor*[^V]{\xi}{}\otimes\omega\mapsto\xi\omega
\end{align*}
\begin{Def*}
The \emph{$i$-th Nygaard complex} of $X/R$ is
\begin{equation*}
\mathcal{N}^{i}W\Omega_{X/R}^{\bullet}:=\varprojlim_{n}\mathcal{N}^{i}W_{n}\Omega_{X/R}^{\bullet}
\end{equation*}
\end{Def*}
Now suppose that $R$ is a ring in which $p$ is nilpotent. Then the hypercohomology of the relative de Rham-Witt complex computes crystalline cohomology
\begin{equation*}
\mathbb{H}^{n}(X,W\Omega_{X/R}^{\bullet})\cong H_{\cris}^{n}(X/W(R))
\end{equation*}
by Theorem \ref{comparison}. Write $\iota_{i},\alpha_{i},F_{i}$ for the maps on hypercohomology induced by the morphisms $\hat{\iota}_{i},\hat{\alpha}_{i},\hat{F}_{i}$. Then the data $(\mathbb{H}^{n}(X,\mathcal{N}^{i}W\Omega_{X/R}^{\bullet}),\iota_{i},\alpha_{i},F_{i})$ is a predisplay structure on $P_{0}=H_{\cris}^{n}(X/W(R))$ and we seek a natural condition on $X$ to ensure that this is a display structure.
\par
Let $A$ be a torsion-free $p$-adic PD-thickening of $R$ and $\mathscr{A}=(A,J,R,\sigma,\dot{\sigma})$ a frame for $R$. Write $A_{n}:=A/p^{n}A$. Suppose that $X/R$ admits a compatible system of smooth and projective liftings $(X_{n}/A_{n})_{n\in\mathbb{N}}$ and suppose moreover that for each $n\geq 0$, $X_{n}$ satisfies the following condition:-
\newline
\begin{align*}
& \text{(LZ1) } \ H^{q}(X_{n},\Omega_{X_{n}/A_{n}}^{p}) \text{ is a locally free finitely generated }A_{n}\text{-module} \\
& \ \ \ \ \ \ \ \ \ \ \text{for all }p,q. \\
& \text{(LZ2) \ The Hodge-de Rham spectral sequence}\\
& \ \ \ \ \ \ \ \ \ \ \ \ \ \ \ \ \ \ \ \ \ \ \ \ \ \ E_{1}^{p,q}=H^{q}(X_{n},\Omega_{X_{n}/A_{n}}^{p})\Rightarrow H_{\dR}^{p+q}(X_{n}/A_{n}) \\
& \ \ \ \ \ \ \ \ \ \ \text{degenerates at }E_{1}.
\end{align*}
\begin{Theorem}\label{Langer-Zink display} 
Let $R$ be a noetherian ring and $p$ be nilpotent in $R$. Then:-
\
\par 
(i) Let $X/R$ be a smooth projective scheme with a with a compatible system of liftings  $(X_{n}/W_{n}(R))_{n\in\mathbb{N}}$ satisfying the analogue of (LZ1) and (LZ2). Then for $m<p$ the data $(P_{i}=\mathbb{H}^{m}(X,\mathcal{N}^{i}W\Omega_{X/R}^{\bullet}),\iota_{i},\alpha_{i}, F_{i})$ is a display-structure $\mathscr{D}_{R}(X)$ on $H_{\mathrm{cris}}^{m}(X/W(R))$. If $R$ is a reduced ring, then we have isomorphisms
\begin{equation*}
P_{i}=\mathbb{H}^{m}(X,\mathcal{N}^{i}W\Omega_{X/R}^{\bullet})\cong H^{m}_{\mathrm{cris}}(X,\mathcal{J}_{X/W(R)}^{[i]})
\end{equation*}
where $\mathcal{J}_{X/W(R)}^{[i]}$ is the $i$-th divided power of the usual crystalline ideal sheaf $\mathcal{J}_{X/W(R)}:=\ker(\mathcal{O}_{X/W(R)}\rightarrow\mathcal{O}_{X})$.
\par
(ii) Let $X/R$ be a smooth projective scheme with a smooth formal lifting  $\mathcal{X}/\text{Spf }A$. For each $n\in\mathbb{N}$ write $X_{n}:=\mathcal{X}\times_{\text{Spf }A}\text{Spec }A_{n}$ and suppose that the system $(X_{n}/A_{n})_{n\in\mathbb{N}}$ satisfies (LZ1) and (LZ2) as above. Then, for $m<p$, there is an $m$-$\mathscr{A}$-window structure (see \cite{LZ07} \S5 for the definition) on $H_{\mathrm{cris}}^{m}(X/A)$, with underlying $A$-modules
\begin{equation*}
P_{i}:=H^{m}_{\mathrm{cris}}(X,\mathcal{J}_{X/A}^{[i]})
\end{equation*}
and divided Frobenius maps $F_{i}:=\frac{1}{p^{i}}F$, where $F$ is the composite 
\begin{equation*}
H_{\mathrm{cris}}^{m}(X/A,\mathcal{J}_{X/A}^{[i]})\rightarrow H_{\mathrm{cris}}^{m}(X/A)\xrightarrow{F} H_{\mathrm{cris}}^{m}(X/A)
\end{equation*}
The base-change of the the $m$-$\mathscr{A}$-window on $H^{m}_{\mathrm{cris}}(X/A)$ along the frame homomorphism $\mathscr{A}\rightarrow\mathscr{W}_{R}$ is the $m$-$\mathscr{W}_{R}$-display $\mathscr{D}_{R}(X)$ of part (i).
\end{Theorem}
\begin{proof}
In this generality, this is (\cite{GL21}, 1.1), although it was shown first for reduced rings $R$ in (\cite{LZ07}, \S5). 
\end{proof}
Let $X/R$ be as above and $m<p$. Let $\mathcal{F}^{r}\Omega^{\bullet}_{X_{n}/W_{n}(R)}$ denote the complex 
\begin{equation*}
I_{R}\otimes\mathcal{O}_{X_{n}}\xrightarrow{pd}I_{R}\otimes\Omega^{1}_{X_{n}/W_{n}(R)}\xrightarrow{pd}\cdots\xrightarrow{pd}I_{R}\otimes\Omega^{r-1}_{X_{n}/W_{n}(R)}\xrightarrow{d}\Omega^{r}_{X_{n}/W_{n}(R)}\xrightarrow{d}\cdots
\end{equation*}
Then $\mathcal{F}^{r}\Omega^{\bullet}_{X_{n}/W_{n}(R)}$ and $\mathcal{N}^{i}W_{n}\Omega_{X/R}^{\bullet}$ are quasi-isomorphic for $r<p$ by \cite[Theorem 1.2]{Lan18}, and the proof shows that these quasi-isomorphisms are compatible for varying $n$, so we get a quasi-isomorphism of pro-complexes $\mathcal{F}^{r}\Omega^{\bullet}_{X_{\bullet}/W_{\bullet}(R)}\cong\mathcal{N}^{i}W_{\bullet}\Omega_{X/R}^{\bullet}$. By \cite[Proposition 3.2]{LZ07}, there is a degenerating spectral sequence $E_{1}^{i,j}\Rightarrow\mathbb{H}^{i+j}(X_{n},\mathcal{F}^{r}\Omega^{\bullet}_{X_{n}/W_{n}(R)})$ where
\begin{equation*}
E_{1}^{i,j}=\left\{
\begin{array}{cc}
I_{R}\otimes H^{j}(X_{n},\Omega_{X_{n}/W_{n}(R)}^{i}) & \text{for }i<r \\
H^{j}(X_{n},\Omega_{X_{n}/W_{n}(R)}^{i}) & \text{for }i\geq r
\end{array}\right.
\end{equation*}
which induces an isomorphism 
\begin{equation*}
P_{r}=\mathbb{H}^{m}(X,\mathcal{F}^{r}\Omega^{\bullet}_{X_{\bullet}/W_{\bullet}(R)})\cong I_{R}L_{0}\oplus I_{R}L_{1}\oplus\cdots\oplus I_{R}L_{r-1}\oplus L_{r}\oplus\cdots\oplus L_{m}
\end{equation*}
where $L_{i}=H^{m-i}(X,\Omega^{i}_{X_{\bullet}/W_{\bullet}(R)})$. These are the $L_{i}$ appearing in a standard datum of the $\mathcal{W}_{R}$-display $\mathscr{D}_{R}(X)$ on $H^{m}_{\mathrm{cris}}(X/W(R))$. In particular, we see that the Hodge filtration of $\mathscr{D}_{R}(X)$ is identified with the Hodge filtration of $X$:-
\begin{equation*}
\text{Fil}^{\bullet}\mathscr{D}_{R}(X)=\text{Fil}^{\bullet}H_{\dR}^{m}(X/R).
\end{equation*}
\section{Crystals of relative displays}\label{Crystals of relative displays section}
Let $k$ be a perfect field of characteristic $p>0$ and $W=W(k)$. Let $R\in\underline{\mathcal{A}rt}_{W,k}$, where we write $\underline{\mathcal{A}rt}_{W,k}$ for the category of local artinian $W$-algebras with residue field $k$. Instructive examples of objects in $\underline{\mathcal{A}rt}_{W,k}$ are $W_{n}(k)$ and $k[t]/(t^{n})$. Let $X_{0}$ be a smooth projective variety with smooth versal deformation space $\mathfrak{S}$ and versal family $\mathfrak{X}/\mathfrak{S}$. Then $\mathfrak{S}\cong\text{Spf }A$, where $A=W(k)\llbracket t_{1},\ldots,t_{h}\rrbracket$ and $h=\dim_{k}H^{1}(X_{0},\mathcal{T}_{X_{0}/k})$. Suppose that for $f:X\rightarrow\text{Spec }R$ a deformation of $X_{0}/k$, the following two conditions analogous to (LZ1) and (LZ2) of Section \ref{display on crystalline} hold:-
\newline
\begin{align*}
& \bullet \ R^{q}f_{\ast}\Omega_{X/R}^{p} \text{ is a free }R\text{-module and commutes with base-change }R\rightarrow R' \\
& \ \ \ \text{ in }\underline{\mathcal{A}rt}_{W,k}\text{, for all }p,q. \\
& \bullet \ \text{The spectral sequence}\\
& \ \ \ \ \ \ \ \ \ \ \ \ \ \ \ \ \ \ \ \ \ \ \ \ \ \ E_{1}^{p,q}=R^{q}f_{\ast}\Omega_{X/R}^{p}\Rightarrow\mathbb{R}^{p+q}f_{\ast}\Omega_{X/R}^{\bullet} \\
& \ \ \ \ \ \ \ \ \ \ \text{degenerates at }E_{1}.
\end{align*}

Consider the Frobenius $\sigma:A\rightarrow A$ given by $\sigma=\tensor*[^{F}]{}{}$ on $W(k)$ and by $T_{i}\mapsto T_{i}^{p}$. Set $C_{m}:=A/(T_{1}^{m},\cdots,T_{h}^{m})$ and $R_{m}:=C_{m}/p^{m}C_{m}$, and write $\sigma:C_{m}\rightarrow C_{m}$ for the endomorphism induced by $\sigma$. Then $(C_{m},p^{m}C_{m},R_{m},\sigma,\sigma/p)$ is a frame for $R_{m}$. So, by Theorem \ref{Langer-Zink display} we have a $n$-$\mathscr{W}_{R_{m}}$-display $\mathscr{D}_{R_{m}}(\mathfrak{X}_{R_{m}})$ on $H_{\cris}^{n}(\mathfrak{X}_{R_{m}}/W(R_{m}))$, for $n<p$. For $m$ large enough, there is a homomorphism $R_{m}\rightarrow R$, and hence a frame homomorphism $\mathscr{W}_{R_{m}}\rightarrow\mathscr{W}_{R}$. Base-changing along this frame homomorphism gives the $n$-$\mathscr{W}_{R}$-display structure $\mathscr{D}_{R}(X)$ on $H_{\cris}^{n}(X/W(R))$ on the deformation $X/R$ of $X_{0}/k$ corresponding to $\text{Spec }R\rightarrow\mathfrak{S}$.  
\begin{Theorem}\label{crystal}
Suppose that $X'/S$ is a deformation of $X/R$ over a PD-thickening $S\twoheadrightarrow R$. Write 
\begin{equation*}
u:\mathcal{W}_{S}\rightarrow\mathcal{W}_{S/R}
\end{equation*} 
for the frame homomorphism. Then $u_{\bullet}\mathscr{D}_{S}(X')$ depends only on $X$ and $S$. 
\end{Theorem}
\begin{proof}
This is (\cite{GL21}. Theorem 1.2).
\end{proof}
\begin{Def*}
We define $\mathscr{D}_{S/R}(X):=u_{\bullet}\mathscr{D}_{S}(X')$, where $X'$ is any deformation over $\mathrm{Spec} \ S$ of $X$. 
\end{Def*}
Then Theorem \ref{crystal} says that the collection of $\mathscr{D}_{S/R}(X)$ for all PD-thickenings $S\twoheadrightarrow R$ defines a crystal of relative displays.

The rest of (\cite{GL21}, Thm. 1.2) proves that, much like $\mathscr{D}_{R}(X)$, the relative display $\mathscr{D}_{S/R}(X)$ has an explicit description in terms of a relative version of the Nygaard complexes. Indeed, let $\mathcal{N}_{rel/R}^{i}W\Omega_{X'/S}^{\bullet}$ be defined as follows
\begin{align*}
W\mathcal{O}_{X'}\oplus\tilde{\mathfrak{a}}^{i}W\mathcal{O}_{X'}\xrightarrow{d\oplus d}
& W\Omega_{X'/S}^{1}\oplus\tilde{\mathfrak{a}}^{i-1}W\Omega_{X'/S}^{1}\xrightarrow{d\oplus d}\cdots \\
& \cdots\xrightarrow{d\oplus d}W\Omega_{X'/S}^{i-1}\oplus\tilde{\mathfrak{a}}W\Omega_{X'/S}^{i-1}\xrightarrow{dV+d}W\Omega_{X'/S}^{i}\xrightarrow{d}\cdots
\end{align*}
where $\tilde{\mathfrak{a}}:=\log^{-1}(\mathfrak{a})$ is the logarithmic Teichm\"{u}ller ideal of $\mathfrak{a}:=\ker(S\twoheadrightarrow R)$. Then it is shown that, under the analogous conditions (LZ1) and (LZ2) for $X'$, the underlying $W(S)$-modules of $\mathscr{D}_{S/R}(X)$ are 
\begin{equation*}
\tilde{P}_{i}=\mathbb{H}^{n}(X',\mathcal{N}_{rel/R}^{i}W\Omega_{X'/S}^{\bullet})
\end{equation*}
The maps $F_{i},\iota_{i},\alpha_{i}$ are induced by the obvious morphisms of complexes
\begin{align*}
& \hat{\alpha}_{i}:\mathcal{J}\otimes\mathcal{N}_{rel/R}^{i}W\Omega_{X'/S}^{\bullet}\rightarrow\mathcal{N}_{rel/R}^{i+1}W\Omega_{X'/S}^{\bullet} \\
& \hat{\iota}_{i}:\mathcal{N}_{rel/R}^{i+1}W\Omega_{X'/S}^{\bullet}\rightarrow\mathcal{N}_{rel/R}^{i}W\Omega_{X'/S}^{\bullet} \\
& \hat{F}_{i}:\mathcal{N}_{rel/R}^{i}W\Omega_{X'/S}^{\bullet}\rightarrow W\Omega_{X'/S}^{\bullet}
\end{align*}
where multiplication by $p$ on $W\Omega_{X'/S}^{j}$ is replaced by the map
\begin{equation*}
\pi:W\Omega_{X'/S}^{j}\oplus\tilde{\mathfrak{a}}^{i-j}W\Omega_{X'/S}^{j}\rightarrow W\Omega_{X'/S}^{j}\oplus\tilde{\mathfrak{a}}^{i-j-1}W\Omega_{X'/S}^{j}
\end{equation*}
which is multiplication by $p$ on $W\Omega_{X'/S}^{j}$ and the inclusion on the other summand. The divided Frobenius $\hat{F}_{r}$ is defined on the subcomplex $\mathcal{N}^{i}W\Omega_{X'/S}^{\bullet}$ as before and on $\tilde{\mathfrak{a}}^{i-j}W\Omega_{X'/S}^{j}$ it is defined to be the zero map.

It is then easy to see from this explicit description in terms of hypercohomology of relative Nygaard complexes that the Hodge filtration of the relative display $\mathscr{D}_{S/R}(X)$ coincides with the Hodge filtration of $X$:-
\begin{equation*}
\text{Fil}^{\bullet}\mathscr{D}_{S/R}(X)=\text{Fil}^{\bullet}H_{\dR}^{n}(X/R).
\end{equation*}

\section{Deformation theory of Calabi-Yau threefolds}\label{Deformation theory of Calabi-Yau threefolds section}
Let $k$ be a perfect field of characteristic $p>0$ and $W=W(k)$. Let $X_{0}$ be a Calabi-Yau threefold over $\mathrm{Spec} \ k$, i.e. a smooth projective $k$-variety with trivial canonical bundle and $H^{1}(X_{0},\mathcal{O}_{X_{0}})=H^{2}(X_{0},\mathcal{O}_{X_{0}})=0$. 
\par
There are two immediate consequences of the definition. First, triviality of the canonical bundle gives the following identification of the tangent sheaf
\begin{equation*}
\Omega_{X_{0}/k}^{2}\cong\Omega_{X_{0}/k}^{3}\wedge\Omega_{X_{0}/k}^{-1}\cong\mathcal{O}_{X_{0}}\wedge\Omega_{X_{0}/k}^{-1}\cong\mathcal{T}_{X_{0}/k}
\end{equation*}
Second, Serre duality along with triviality of the canonical bundle shows that the Hodge diamond of $X_{0}$ is as follows:-
\begin{center}
\begin{tikzpicture}[descr/.style={fill=white,inner sep=1.5pt}]
        \matrix (m) [
            matrix of math nodes,
            row sep=0.2em,
            column sep=0.5em,
            text height=1.5ex, text depth=0.25ex
        ]
        { & & & 1 & & & \\
          & & h^{1,0} &  & 0 & & \\
           & h^{2,0} & & h^{1,1} & & 0 & \\
          1 & & h^{1,2} &  & h^{1,2} &  & 1 \\
          & 0 & & h^{1,1} & & h^{2,0} & \\
          & & 0 &  & h^{1,0} & & \\
          & & & 1 & & & \\
        };

\end{tikzpicture} 
\end{center}
where $h^{i,j}:=\dim_{k}H^{j}(X_{0},\Omega_{X_{0}/k}^{j})$. One should note that Hodge symmetry (i.e. $h^{i,j}=h^{j,i}$) fails in positive characteristic, in general. It is clear that the Calabi-Yau threefold $X_{0}$ satisfies Hodge symmetry if and only if $H^{0}(X_{0},\Omega_{X_{0}/k}^{1})=H^{0}(X_{0},\mathcal{T}_{X_{0}/k})=0$. 

We make the assumption that $X_{0}$ has a smooth versal deformation space $\mathfrak{S}$. Then $\mathfrak{S}=\mathrm{Spf}\,A$ where $A=W\llbracket t_{1},\ldots,t_{h}\rrbracket$ and $h=h^{2,1}$. Moreover we assume that the Hodge-de Rham spectral sequence
\begin{equation*}
E_{1}^{i,j}=H^{j}(\mathfrak{X},\Omega_{\mathfrak{X}/A}^{i})\Rightarrow H_{\mathrm{dR}}^{i+j}(\mathfrak{X}/A)
\end{equation*}
degenerates at $E_{1}$ and the $E_{1}$-terms are free $\mathcal{O}_{\mathfrak{S}}$-modules. Note that this implies that the Hodge and de Rham cohomology computes with arbitrary base change.

\begin{Remark}\label{remark}\emph{
We point out that the assumption of the smoothness of the versal deformation space is non-trivial. Indeed, Calabi-Yau threefolds in positive characteristic that have obstructed deformations are constructed in \cite{CvS09}, \cite{Hir99}, \cite{HIS07}, \cite{HIS08}, \cite{Sch09} and \cite{Sch04}. Such examples are rather rare however, and Calabi-Yau threefolds having unobstructed deformations should in fact be ``generic''. Rather than formulate this in a precise sense, let us instead point out that a special case of \cite[Theorem 4.1]{Sch03} shows that if $X_{0}$ admits a smooth proper formal lifting to characteristic zero and has torsion-free crystalline cohomology, then the deformation theory of $X_{0}$ is unobstructed (also see \cite[Theorem C]{ESB05} for a related criterion). It is conjectured in \cite[Conjecture 7.13]{Jos20} that a Calabi-Yau threefold $X_{0}$ over an algebraically closed field $k$ of characteristic $p$ lifts to characteristic zero if and only if $h^{1,0}=0$ and $h_{W}^{1,2}\geq 0$. The first condition is equivalent to $H_{\mathrm{cris}}^{2}(X_{0}/W(k))$ being torsion-free, and is also equivalent to $\mathrm{Pic}(X_{0})$ having no $p$-torsion \cite[Theorem 18]{Tak17}. The non-negativity of the Hodge-Witt number in the second condition is equivalent to $H_{\mathrm{cris}}^{3}(X_{0}/W(k))$ being torsion-free by \cite[Theorem 7.4]{Jos20}. Note that the Artin-Mazur formal group of $X_{0}$ has finite height if and only if $X_{0}$ is Hodge-Witt \cite[Theorem 6.1]{Jos07}, and $X_{0}$ being Hodge-Witt implies $h_{W}^{1,2}\geq 0$ \cite[Theorem 7.4]{Jos20}. If we admit Joshi's conjecture, therefore, we see that any Calabi-Yau threefold with finite height and $p$-torsion-free $\mathrm{Pic}(X_{0})$ will have unobstructed deformations (at least if we work over an algebraically closed field).} 

\emph{Let us point out that one can easily construct examples satisfying the conditions of Schr\"{o}er's theorem by the following standard argument: Let $X/\mathbb{C}$ be a complex Calabi-Yau threefold. By spreading-out we can assume that $X$ is the base-change of a smooth projective family $f:\mathcal{X}\rightarrow S$ where $S$ is an integral flat $\mathbb{Z}$-scheme of finite type. By shrinking $S$ if necessary, we may assume that the sheaves $R^{p}f_{\ast}\Omega_{\mathcal{X}/S}^{p}$ and $\mathbb{R}f_{\ast}\Omega_{\mathcal{X}/S}^{\bullet}$ are locally free of finite type, and hence commute with base change. Then any fibre over a closed point in the smooth locus of $f$ will be a Calabi-Yau threefold $X_{0}$ with the desired properties; $X_{0}$ lifts by definition and the crystalline cohomology is torsion-free by the comparison isomorphism with de Rham cohomology and that $\mathbb{R}f_{\ast}\Omega_{\mathcal{X}/S}^{\bullet}$ is locally free.}

\emph{The assumptions of Hodge-de Rham degeneration and freeness of the Hodge cohomology groups are also non-trivial. Note that they satisfied for families of smooth complete intersections by \cite[Th\'{e}or\`{e}me 1.5]{Del73}. Since the deformations of a complete intersection are unobstructed (see e.g. \cite[Theorem 9.4]{Har10}),  complete intersection Calabi-Yau threefolds (so-called CICY's) satisfy all of our hypotheses. Note that $h^{2,0}=0$ for a CICY by \cite[Th\'{e}orem\`{e} 2.3]{Del73}, so in this case the deformation space $\mathfrak{S}$ is universal.}

\end{Remark}

Write $\mathfrak{X}\rightarrow\mathfrak{S}$ for the versal family and let $\mathfrak{X}_{n}:=\mathfrak{X}\times_{A}A/p^{n+1}A$ for $n\geq 0$. Then
\begin{equation*}
H^{i}(\mathfrak{X}_{0},\mathcal{O}_{\mathfrak{X}_{0}})\otimes_{A}k=H^{i}(X_{0},\mathcal{O}_{X_{0}})=0
\end{equation*}
and therefore $H^{1}(\mathfrak{X}_{0},\mathcal{O}_{\mathfrak{X}_{0}})=H^{2}(\mathfrak{X}_{0},\mathcal{O}_{\mathfrak{X}_{0}})=0$ by Nakayama's Lemma. From the long exact sequence on cohomology we have
\begin{equation*}
H^{1}(\mathfrak{X}_{0},\mathcal{O}_{\mathfrak{X}_{0}})\rightarrow H^{1}(\mathfrak{X}_{n+1},\mathcal{O}_{\mathfrak{X}_{n+1}}^{\ast})\rightarrow H^{1}(\mathfrak{X}_{n},\mathcal{O}_{\mathfrak{X}_{n}}^{\ast})\rightarrow H^{2}(\mathfrak{X}_{0},\mathcal{O}_{\mathfrak{X}_{0}})
\end{equation*}
and we conclude that the natural restriction $\text{Pic}(X_{n+1})\rightarrow\text{Pic}(X_{n})$ is an isomorphism. For each $n$ let $\omega_{X_{n}}$ be a $3$-form. Since ${\omega_{X_{n+1}}|}_{X_{n}}\simeq\omega_{X_{n}}$, induction gives $\Omega^{3}_{\mathfrak{X}/A}\simeq\mathcal{O}_{\mathfrak{X}}$, and so we see that $\mathcal{T}_{\mathfrak{X}/\mathfrak{S}}\simeq\Omega^{2}_{\mathfrak{X}/A}$ on $\mathfrak{X}$.  
\par
We denote by $\fil^{i}H_{\dR}^{n}(\mathfrak{X}/A)$ the $i^{\text{th}}$-step in the Hodge filtration of $H_{\dR}^{n}(\mathfrak{X}/A)$, and write $\text{gr}^{i}H_{\dR}^{n}(\mathfrak{X}/A)$ for the subquotients of the filtration, i.e.
\begin{equation*}
\text{gr}^{i}H_{\dR}^{n}(\mathfrak{X}/A):=\frac{\fil^{i}H_{\dR}^{n}(\mathfrak{X}/A)}{\fil^{i+1}H_{\dR}^{n}(\mathfrak{X}/A)}
\end{equation*}
Our assumption on that the Hodge-de Rham spectral sequence degenerates implies that we have isomorphisms
\begin{equation*}
H^{n-i}(\mathfrak{X},\Omega_{\mathfrak{X}/A}^{i})\xrightarrow{\sim}\text{gr}^{i}H_{\dR}^{n}(\mathfrak{X}/A)
\end{equation*} 
for all $i,n$. Consider the Gauss-Manin connection 
\begin{equation*}
\nabla:H_{\dR}^{n}(\mathfrak{X}/A)\rightarrow H_{\dR}^{n}(\mathfrak{X}/A)\otimes_{A}\Omega^{1}_{A/W}
\end{equation*}
Then Griffiths transversality states that
\begin{equation*}
\nabla\left(\fil^{i}H_{\dR}^{n}(\mathfrak{X}/A)\right)\subseteq\fil^{i-1}H_{\dR}^{n}(\mathfrak{X}/A)\otimes_{A}\Omega^{1}_{A/W}
\end{equation*}
and hence $\nabla$ induces an $A$-linear map:-
\begin{equation*}
\text{gr}^{i}\nabla:\text{gr}^{i}H_{\dR}^{n}(\mathfrak{X}/A)\rightarrow\text{gr}^{i-1}H_{\dR}^{n}(\mathfrak{X}/A)\otimes_{A}\Omega^{1}_{A/W}
\end{equation*}
Taking duals gives an $A$-linear map
\begin{equation*}
\mathcal{T}_{\mathfrak{S}/W}\rightarrow\text{Hom}_{A}\left(\text{gr}^{i}H_{\dR}^{n}(\mathfrak{X}/A),\text{gr}^{i-1}H_{\dR}^{n}(\mathfrak{X}/A)\right)
\end{equation*}
In particular, putting $n=i=3$ gives
\begin{equation}\label{3}
\mathcal{T}_{\mathfrak{S}/W}\rightarrow\text{Hom}_{A}\left(H^{0}(\mathfrak{X},\mathcal{O}_{\mathfrak{X}}),H^{1}(\mathfrak{X},\mathcal{T}_{\mathfrak{X}/\mathfrak{S}})\right)
\end{equation}
\begin{Proposition}\label{Kodaira-Spencer}
The map (\ref{3}) is surjective.
\end{Proposition}
\begin{proof}
This is essentially the derivative of the period map. It is well-known (e.g. \cite{Voi02}, 10.4) that the map (\ref{3}) is the composition of the cup-product with the Kodaira-Spencer map:-
\begin{equation*}
\text{KS}(\mathfrak{X}/\mathfrak{S}):\mathcal{T}_{\mathfrak{S}/W}\rightarrow H^{1}(\mathfrak{X},\mathcal{T}_{\mathfrak{X}/\mathfrak{S}})
\end{equation*}
Since the formal deformation is versal, $\text{Kod}(\mathfrak{X}/\mathfrak{S})$ is a surjection. Thus we are in the following situation
\begin{center}
\begin{tikzpicture}
 \node (a) {$\mathcal{T}_{\mathfrak{S}/W}$};
\node (b) [right =3cm of a]{$H^{1}(\mathfrak{X},\mathcal{T}_{\mathfrak{X}/\mathfrak{S}})$};
\node (c) [below =1.5cm of b]{$\text{Hom}_{A}\left(H^{0}(\mathfrak{X},\mathcal{O}_{\mathfrak{X}}),H^{1}(\mathfrak{X},\mathcal{T}_{\mathfrak{X}/\mathfrak{S}})\right)$};
  \draw[->>](a) to node [above]{$\text{KS}(\mathfrak{X}/\mathfrak{S})$}(b);
\draw[->](b) to node {}(c);
\draw[->](a) to node {}(c);
\end{tikzpicture}
\end{center} 
where the right-hand arrow is given by cup-product. But choosing a generator $\omega\in H^{0}(\mathfrak{X},\mathcal{O}_{\mathfrak{X}})$ and using that cup-product is a perfect pairing shows that this is surjective. 
\end{proof}
Composing $\nabla$ with the natural maps
\begin{equation*}
\partial/\partial t_{i}:\Omega^{1}_{A/W}\rightarrow A \ \ \ \ \ (i=1,\ldots,h)
\end{equation*}
gives maps 
\begin{equation*}
\nabla_{i}:H^{3}_{\mathrm{dR}}(\mathfrak{X}/A)\rightarrow H^{3}_{\mathrm{dR}}(\mathfrak{X}/A)
\end{equation*}
and Proposition \ref{Kodaira-Spencer} shows that $\{\nabla_{1}(\omega),\ldots,\nabla_{h}(\omega)\}$ generates $H^{1}(\mathfrak{X},\mathcal{T}_{\mathfrak{X}/\mathfrak{S}})$ as an $A$-module, for $\omega\in H^{0}(\mathfrak{X},\mathcal{O}_{\mathfrak{X}})$.
\newline
\par
We will now describe the deformation theory of $X_{0}$ over nilpotent PD-thickenings. Let $\alpha:(S,\mathfrak{m}_{S})\twoheadrightarrow(R,\mathfrak{m}_{R})$ be a surjection in $\underline{\mathcal{A}rt}_{W,k}$ with $\mathfrak{a}\mathfrak{m}_{S}=0$, where $\mathfrak{a}=\ker\alpha$ (a so-called \emph{`small surjection'}).  Let $X$ be a lifting of $X_{0}/k$ to $R$. For now, endow $\alpha$ with the trivial PD-structure.
and let $X',Y$ be deformations of $X/R$ over $\text{Spec }S$. Then, since $\mathfrak{X}\rightarrow\mathfrak{S}=\text{Spf }A$ is the versal family, $X'$ and $Y$ are given by $W$-algebra homomorphisms $f,g:A\rightarrow S$ respectively, and we have a commutative diagram:-
\begin{center}
\begin{tikzpicture}
 \node (a) {$A$};
\node (b) [right =1cm of a]{$S$};
\node (c) [right =1cm of b]{$R$};
\path[->]
([yshift= 2pt]a.east) edge node[above] {$f$} ([yshift= 2pt]b.west)
([yshift= -2pt]a.east) edge node[below] {$g$} ([yshift= -2pt]b.west);
 \draw[->>](b) to node [above]{$\alpha$}(c);
\end{tikzpicture}
\end{center} 
Given $u\in H_{\dR}^{3}(X'/S)$, let $\tilde{u}\in H_{\dR}^{3}(\mathfrak{X}/A)$ be such that $f^{\ast}(\tilde{u})=u$, and set $v=g^{\ast}(\tilde{u})$. Then the Gauss-Manin connection induces an isomorphism 
\begin{align*}\tag{4}\label{4}
& \Psi:H_{\dR}^{3}(X'/S)\xrightarrow{\sim}H_{\dR}^{3}(Y/S)\\
& \ \ \ \ \ \ \ \ u\mapsto v+\sum_{i=1}^{h}(f(t_{i})-g(t_{i}))\breve{\nabla}_{i}(\tilde{u})
\end{align*} 
where $\breve{\nabla}_{i}(\tilde{u})$ denotes the image of $\nabla_{i}(\tilde{u})$ in $H_{\dR}^{3}(X_{0}/k)$ (see \cite{Del81}).
\par
Write $F'=\fil^{3}H_{\dR}^{3}(X'/S)$. Let $F\subset H_{\dR}^{3}(X'/S)$ be a \emph{``lifting of the Hodge filtration''} on $X$, i.e. a direct summand lifting $\fil^{3}H_{\dR}^{3}(X/R)$. Since the canonical map
\begin{equation*}
F\rightarrow H_{\dR}^{3}(X'/S)/F'
\end{equation*}
factors through $F\rightarrow\fil^{3}H_{\dR}^{3}(X_{0}/k)=H^{0}(X_{0},\Omega^{3}_{X_{0}/k})$ and has image in
\begin{equation*}
\mathfrak{a}\left(H_{\dR}^{3}(X'/S)/F'\right)\cong\mathfrak{a}\otimes_{k}\left(H_{\dR}^{3}(X_{0}/k)/\fil^{3}H_{\dR}^{3}(X_{0}/k)\right)
\end{equation*}
we see that liftings of the Hodge filtration are classified by $k$-linear maps
\begin{equation*}
\varkappa(F):H^{0}(X_{0},\Omega^{3}_{X_{0}/k})\rightarrow \mathfrak{a}\otimes_{k}\left(H_{\dR}^{3}(X_{0}/k)/\fil^{3}H_{\dR}^{3}(X_{0}/k)\right)
\end{equation*} 
\begin{Proposition}\label{inverse image is in fil2}
We have that $\Psi^{-1}(\emph{\fil}^{3}H_{\emph{\dR}}^{3}(Y/S))\subseteq\emph{\fil}^{2}H_{\emph{\dR}}^{3}(X'/S)$.
\end{Proposition}
\begin{proof}
Set $F_{Y}:=\Psi^{-1}(\fil^{3}H_{\dR}^{3}(Y/S)$. By the above discussion, it is equivalent to show that $\varkappa(F_{Y})$ has image contained in $\mathfrak{a}\otimes_{k}\text{gr}^{2}H_{\dR}^{3}(X_{0}/k)$. We will do this using the explicit expression given in the isomorphism (\ref{4}):-
\par
As above, let $u\in H_{\dR}^{3}(X'/S)$, choose a generator $\tilde{u}\in\fil^{3}H_{\dR}^{3}(\mathfrak{X}/A)$ such that $f_{\ast}(\tilde{u})=u$, and write $u_{0}$ for the image of $\tilde{u}$ in $H^{0}(X_{0},\Omega^{3}_{X_{0}/k})$. Then the isomorphism (\ref{4}) shows that
\begin{equation*}
u-\sum_{i=1}^{r}(f(t_{i})-g(t_{i}))\breve{\nabla}_{i}(\tilde{u})
\end{equation*}
is a generator of $F_{Y}$. Then (\cite{Del81}, 1.1.2) says that the map $\varkappa(F_{Y})$ on $u_{0}$ is given by
\begin{equation*}
\varkappa(F_{Y})(u_{0})=-\sum_{i=1}^{r}(f(t_{i})-g(t_{i}))\otimes\breve{\nabla}_{i}(\tilde{u})\in\mathfrak{a}\otimes_{k}\text{gr}^{2}H_{\dR}^{3}(X_{0}/k) 
\end{equation*}
since the $\breve{\nabla}_{i}(\tilde{u})$ are in $H^{1}(X_{0},\mathcal{T}_{X_{0}/k})=\text{gr}^{2}H_{\dR}^{3}(X_{0}/k)$.
\end{proof}
\begin{Proposition}\label{bijection between lines and deformations}
In the above situation, we have an identification
\begin{align*}
& \text{\emph{Def}}_{X}(S)=\left\{\text{lines }E\subseteq\emph{\fil}^{2}H_{\mathrm{dR}}^{3}(X'/S)\text{ lifting }\emph{\fil}^{3}H_{\mathrm{dR}}^{3}(X/R)\right\} \\
& \ \ \ \ \ \  Y \ \ \ \mapsto \ \ \ F_{Y}:=\Psi^{-1}(\emph{\fil}^{3}H_{\mathrm{dR}}^{3}(Y/S))
\end{align*}
where $\mathrm{Def}_{X}(S)$ denotes the set of isomorphism classes of deformations of $X$ over $\mathrm{Spec}\,S$.
\end{Proposition}
\begin{proof}
Proposition \ref{inverse image is in fil2} exactly tells us that $F_{Y}$ is in the right-hand side. We provide an inverse map:- 
\par
As was already noted, Proposition \ref{Kodaira-Spencer} shows that the $\breve{\nabla}_{i}(\tilde{u})$ form a basis of $\text{gr}^{2}H_{\dR}^{3}(X_{0}/k)$. Therefore, $F_{Y}$ determines $a_{i}\in\mathfrak{a}$ ($i=1,\ldots,r$) such that 
\begin{equation*}
u-\sum_{i=1}^{r}a_{i}\breve{\nabla}_{i}(\tilde{u})
\end{equation*}
generates $F_{Y}$. Then we get a Calabi-Yau threefold $Y$ over $\mathrm{Spec}\,S$ by pulling back the versal family along the $W$-algebra homomorphism
\begin{align*}
& g:A\rightarrow S\\
& \ \ \ \ t_{i}\mapsto f(t_{i})-a_{i}
\end{align*}
To see that the map $F_{Y}\mapsto Y$ just constructed is indeed the inverse of the map in the statement of the proposition, first note that $\Psi^{-1}(\mathrm{Fil}^{3}H^{3}_{\mathrm{dR}}(Y/S))$ is generated by 
\begin{equation*}
u-\displaystyle\sum_{i}^{r}(f(t_{i})-g(t_{i}))\breve{\nabla}_{i}(\tilde{u})=u-\displaystyle\sum_{i}^{r}(f(t_{i})-(f(t_{i})-a_{i}))\breve{\nabla}_{i}(\tilde{u})=u-\displaystyle\sum_{i}^{r}a_{i}\breve{\nabla}_{i}(\tilde{u})
\end{equation*}
and hence $\Psi^{-1}(\mathrm{Fil}^{3}H^{3}_{\mathrm{dR}}(Y/S))=F_{Y}$. This shows that one composition of the two maps is the identity. To see that the reverse composition is the identity, let $Y$ be the a deformation over $\mathrm{Spec}\,S$ of $X$. Then $Y$ is the pullback of the versal family $\mathfrak{X}$ along some $W$-algebra homomorphism $g:A\rightarrow S$. We have seen already that $F_{Y}$ is generated by 
\begin{equation*}
u-\displaystyle\sum_{i}^{r}(f(t_{i})-g(t_{i}))\breve{\nabla}_{i}(\tilde{u})
\end{equation*}
and therefore the deformation associated to $F_{Y}$ is, by definition, $\mathfrak{X}\times_{\mathrm{Spf}\,A}\mathrm{Spec}\,S$ where the homomorphism $A\rightarrow S$ is given by $t_{i}\mapsto f(t_{i})-(f(t_{i})-g(t_{i}))=g(t_{i})$, that is $\mathfrak{X}\times_{\mathrm{Spf}\,A}\mathrm{Spec}\,S=Y$.

\end{proof}
In the following, given a line $E\subset H^{3}_{\mathrm{dR}}(X'/S)$ we shall write $\nabla E$ for the $S$-module generated by $\nabla_{1}(e),\ldots,\nabla_{h}(e)$, where $e$ is a generator of $E$.

\begin{Lemma}\label{no perp}
Suppose that $\mathfrak{a}^{2}=0$. Let $E\subset H^{3}_{\mathrm{dR}}(X'/S)$ be a line lifting $\mathrm{Fil}^{3}H^{3}_{\mathrm{dR}}(X/R)$. If $E\subset\mathrm{Fil}^{2}H^{3}_{\mathrm{dR}}(X'/S)$ then $(E\oplus\nabla E)^{\perp}=E\oplus\nabla E$ under the cup-product.
\end{Lemma}
\begin{proof}
$E$ is generated by an element of the form
\begin{equation*}
u+\sum_{i=1}^{h}\alpha_{i}\breve{\nabla}_{i}(\tilde{u})
\end{equation*}
for some $\alpha_{1},\ldots,\alpha_{h}\in\mathfrak{a}$. Since the cup-product $\langle -,-\rangle$ is anti-symmetric on $H^{3}_{\mathrm{dR}}(X'/S)$, to see that $E\oplus\nabla E=(E\oplus\nabla E)^{\perp}$ it suffices to show that 
\begin{equation}\label{first}
\langle u+\sum_{i=1}^{h}\alpha_{i}\breve{\nabla}_{i}(\tilde{u}),\nabla_{s}(u)+\sum_{i=1}^{h}\alpha_{i}\nabla_{s}\breve{\nabla}_{i}(\tilde{u})\rangle=0
\end{equation} 
for each $s=1,\ldots, h$ and
\begin{equation}\label{second}
\langle \nabla_{s}(u)+\sum_{i=1}^{h}\alpha_{i}\nabla_{s}\breve{\nabla}_{i}(\tilde{u}),\nabla_{t}(u)+\sum_{i=1}^{h}\alpha_{i}\nabla_{t}\breve{\nabla}_{i}(\tilde{u})\rangle=0
\end{equation} 
for each $s,t=1,\ldots, h$ with $s\neq t$. The only non-trivial cup-product pairings are between $H^{0}(X',\Omega_{X'/S}^{3})$ and $H^{3}(X',\mathcal{O}_{X'})$, and between $H^{1}(X',\Omega_{X'/S}^{2})$ and $H^{2}(X',\Omega_{X'/S}^{1})$. Therefore the left-hand side of \eqref{first} is
\begin{equation*}
\sum_{i,j=1}^{h}\alpha_{i}\alpha_{j}\langle\breve{\nabla}_{i}(\tilde{u}),\nabla_{s}\breve{\nabla}_{i}(\tilde{u})\rangle
\end{equation*}
and the left-hand side of \eqref{second} is 
\begin{equation*}
\sum_{i,j=1}^{h}\alpha_{i}\alpha_{j}\langle\nabla_{s}\breve{\nabla}_{i}(\tilde{u}),\nabla_{t}\breve{\nabla}_{i}(\tilde{u})\rangle
\end{equation*}
Both expressions vanish because $\mathfrak{a}^{2}=0$.
\end{proof}

We now extend Proposition \ref{bijection between lines and deformations} to arbitrary nilpotent PD-thickenings:-
\begin{Theorem}\label{full bijection between lines and deformations}
Let $X_{0}/k$ be a Calabi-Yau threefold and let $\alpha:S\twoheadrightarrow R$ be a surjection in $\underline{\mathcal{A}rt}_{W,k}$ such that $\mathfrak{a}=\ker\alpha$ is endowed with a nilpotent PD-structure which is compatible with the canonical PD-structure on $pW$. Write $\mathfrak{a}^{[n]}$ for the $n^{th}$-divided power ideal of $\mathfrak{a}$ for $n\geq 1$, and $\mathfrak{a}^{[0]}=S$.
\par
Let $X$ be a lifting of $X_{0}/k$ to $R$, and let $X'$ be a deformation of $X/R$ over $\text{Spec }S$. Then if $Y$ is a deformation of $X/R$ over $S$, the Gauss-Manin connection induces an isomorphism
\begin{equation*}
\Psi:H_{\mathrm{dR}}^{3}(X'/S)\xrightarrow{\sim}H_{\mathrm{dR}}^{3}(Y/S)
\end{equation*} 
We have a bijection between $\text{\emph{Def}}_{X}(S)$ and 
\begin{equation*}
\left\{\begin{array}{c}
\text{lines }E\subseteq\sum_{i+j=3}\mathfrak{a}^{[i]}\emph{\fil}^{j}H_{\mathrm{dR}}^{3}(X'/S)\text{ lifting }\emph{\fil}^{3}H_{\mathrm{dR}}^{3}(X/R)\\
\text{ such that }(E\oplus\nabla E)^{\perp}=E\oplus\nabla E
\end{array}\right\}
\end{equation*}
given by $Y\mapsto F_{Y}:=\Psi^{-1}(\emph{\fil}^{3}H_{\mathrm{dR}}^{3}(Y/S))$.
\end{Theorem}
\begin{proof}
The map $Y\mapsto F_{Y}$ is well-defined by \cite[Lemme 1.1.2]{Del81} (see also \cite[Theorem 2.4]{Ogu78}), and the fact that the cup-product is horizontal for the Gauss-Manin connection.

In the usual way, we first reduce the statement by decomposing $\alpha$ into a series of ``smaller'' PD-thickenings:-
\par
Let $t$ be such that
\begin{equation*}
0=\mathfrak{a}^{[t]}\subset\mathfrak{a}^{[t-1]}\subset\cdots\subset\mathfrak{a}^{[2]}\subset\mathfrak{a}
\end{equation*}
Then we can write $\alpha$ as the composition of two nilpotent PD-thickenings:-
\begin{equation*}
\alpha:S\twoheadrightarrow R_{1}:=S/\mathfrak{a}^{[t-1]}\twoheadrightarrow R
\end{equation*}
Notice that, by induction, we can assume the theorem holds for the second PD-thickening (the base case $t=2$ is exactly Proposition \ref{bijection between lines and deformations}).
\par
Now, given a line $E\subseteq\sum_{i+j=3}\mathfrak{a}^{[i]}\fil^{j}H_{\mathrm{dR}}^{3}(X'/S)$ lifting $\fil^{3}H_{\dR}^{3}(X/R)$ such that $(E\oplus\nabla E)^{\perp}=E\oplus\nabla E$, we get an induced $E_{R_{1}}\subseteq\sum_{i+j=3}\mathfrak{a'}^{[i]}\fil^{j}H_{\dR}^{3}(X'\times_{S}R_{1}/R_{1})$, where $\mathfrak{a}'=\ker(R_{1}\twoheadrightarrow R)$, such that $(E_{R_{1}}\oplus\nabla E_{R_{1}})^{\perp}=E_{R_{1}}\oplus\nabla E_{R_{1}}$. But since the theorem holds for $R_{1}\twoheadrightarrow R$, there corresponds to $E_{R_{1}}$ a deformation $Z$ of $X/R$ over $R_{1}$. Choose some deformation $Z'$ of $Z/R_{1}$ over $S$. Then both $X'$ and $Z'$ lift $X$ to $\text{Spec S}$, so we have an isomorphism
\begin{equation*}
H_{\dR}^{3}(X'/S)\xrightarrow{\sim}H_{\dR}^{3}(Z'/S)
\end{equation*}  
Let $e$ be a generator of $E$. Let $\tilde{e}\in H^{3}_{\mathrm{dR}}(\mathfrak{X}/A)$ be such that $f^{\ast}(\tilde{e})=e$ and set $e'=g^{\ast}(\tilde{e})$, where $f,g:A\rightrightarrows S$ are $W$-algebra homomorphisms giving $X'$ and $Z$. Then the image of $e$ under this isomorphism is given by (see \cite[Lemme 1.1.2]{Del81})
\begin{equation*}
e'+\sum_{\underline{m}}(f(\underline{t})-g(\underline{t}))^{[\underline{m}]}\breve{\nabla}^{\underline{m}}(\tilde{e})
\end{equation*}
where the sum is over multi-indices $\underline{m}=(m_{1},\ldots,m_{h})\in\mathbb{N}^{h}$. The notation $(f(\underline{t})-g(\underline{t}))^{[\underline{m}]}$ means 
\begin{equation*}
(f(t_{1})-g(t_{1}))^{[m_{1}]}\cdots(f(t_{h})-g(t_{h}))^{[m_{h}]}
\end{equation*}
and $\breve{\nabla}^{\underline{m}}(\tilde{e})$ is the image in $H^{3}_{\mathrm{dR}}(X_{0}/k)$ of 
\begin{equation*}
\nabla^{\underline{m}}(\tilde{e}):=\nabla_{1}(\tilde{e})^{m_{1}}\cdots\nabla_{h}(\tilde{e})^{m_{h}}
\end{equation*}
The image $E'$ of $E$ under this isomorphism is then a line in $\sum_{i+j=3}\mathfrak{a}^{[i]}\fil^{j}H_{\dR}^{3}(Z'/S)$ lifting $\fil^{3}H_{\dR}^{3}(Z/R_{1})$, and $(E'\oplus\nabla E')^{\perp}=E'\oplus\nabla E'$ because the cup-product is horizontal for the Gauss-Manin connection. In this way we are reduced to proving the theorem for the nilpotent PD-thickening
\begin{equation*}
S\twoheadrightarrow R_{1}=S/\mathfrak{a}^{[t-1]}
\end{equation*}
But the divided powers on $\mathfrak{b}=\ker(S\twoheadrightarrow R_{1})=\mathfrak{a}^{[t-1]}$ are trivial, and $\mathfrak{b}^{2}=0$. Then we can decompose again into $S\twoheadrightarrow R_{m}\twoheadrightarrow\cdots\twoheadrightarrow R_{1}$ and conclude surjectivity by Proposition \ref{bijection between lines and deformations}. It should also be noted that Proposition \ref{bijection between lines and deformations} at each step also gives injectivity. 
\end{proof}

\begin{Def*}
Consider a lifting $E^{\bullet}$ of $\text{Fil}^{\bullet}\mathscr{D}_{S/R}(X)=\text{Fil}^{\bullet}H_{\emph{\dR}}^{3}(X/R)$
\begin{equation*}
E^{3}\subseteq E^{2}\subseteq E^{1}\subseteq\tilde{P}_{0}/I_{S}\tilde{P}_{0}=H_{\mathrm{cris}}^{3}(X/S)=H_{\mathrm{dR}}^{3}(X'/S)
\end{equation*}
and write $\nabla$ for the Gauss-Manin connection on $\tilde{P}_{0}/I_{S}\tilde{P}_{0}$.  We say that $E^{\bullet}$ is of CY-type if the following three conditions hold:-
\begin{align*}
& \text{(i) }E^{3}=\langle e\rangle\subseteq\sum_{i+j=3}\mathfrak{a}^{[i]}\text{Fil}^{j}H_{\mathrm{dR}}^{3}(X'/S)\text{ is a line lifting }\text{Fil}^{3}H_{\mathrm{dR}}^{3}(X/R) \\
& \text{(ii) }E^{1}=\left(E^{3}\right)^{\perp} \\
& \text{(iii) }E^{2}=E^{3}\oplus\nabla E^{3}\text{ and }(E^{2})^{\perp}=E^{2}
\end{align*}
\end{Def*}

We may now rephrase Theorem \ref{full bijection between lines and deformations} in terms of liftings of the Hodge filtration of the associated relative $\mathcal{W}_{S/R}$-display. We assume that the deformation $X$ satisfies the two conditions in \S\ref{Crystals of relative displays section}. For this it is sufficient to assume that $X_{0}$ satisfies the conditions. For example they hold for the Calabi-Yau threefolds given in Remark \ref{remark} by \cite{DI87}, provided that $p\geq 5$. Notice that the Gauss-Manin connection of $\mathfrak{X}/\mathfrak{S}$ gives the Gauss-Manin connection on $H_{dR}^{3}(X'/S')$ by base-change along $A\rightarrow S$.

\begin{Theorem}\label{defs and CY-type liftings}
There is a bijection
\begin{equation*}
\text{\emph{Def}}_{X}(S)\xrightarrow{1:1}\left\{\text{CY-type liftings of Fil}^{\bullet}\mathscr{D}_{S/R}(X)\right\} 
\end{equation*}
\end{Theorem}
\begin{proof}
There is a bijection
\begin{center}
\begin{adjustbox}{width=12cm}
\begin{tikzpicture}
 \node (a) [align=center] {$ \ \text{ direct summands }$ \\ $E\subseteq\sum_{i+j=3}\mathfrak{a}^{[i]}\fil^{j}H_{\mathrm{dR}}^{3}(X'/S) \ $ \\ $\text{ lifting }\text{Fil}^{3}H_{\dR}^{3}(X/R) \text{ such}$ \\ $\text{that } (E\oplus\nabla E)^{\perp}=E\oplus\nabla E$};
\node (b) [align=center] [right =0.8cm of a]{$ \ \text{CY-type liftings of Fil}^{\bullet}\mathscr{D}_{S/R}(X)$};
 \draw[->](a) to node [above]{\emph{\small{1:1}}} node [below]{}(b);
 \draw [decorate,decoration={brace,amplitude=2pt},xshift=0pt,yshift=0pt]
(-2.5,-0.8) -- (-2.5,0.8)node [black,xshift=0pt] {};
 \draw [decorate,decoration={brace,amplitude=2pt},xshift=0pt,yshift=0pt]
(2.35,0.8) -- (2.35,-0.8)node [black,xshift=0pt] {};
 \draw [decorate,decoration={brace,amplitude=2pt},xshift=0pt,yshift=0pt]
(3.6,-0.5) -- (3.6,0.5)node [black,xshift=0pt] {};
\draw [decorate,decoration={brace,amplitude=2pt},xshift=0pt,yshift=0pt]
(8.7,0.5) -- (8.7,-0.5)node [black,xshift=0pt] {};
\end{tikzpicture}
\end{adjustbox}
\end{center} 
given by the obvious map:-
\begin{equation*}
E=\langle e\rangle\mapsto 0\subseteq E\subseteq E\oplus\nabla E\subseteq E^{\perp}\subseteq\tilde{P}_{0}/I_{S}\tilde{P}_{0}= H_{\dR}^{3}(X'/S)
\end{equation*}
It is clearly a bijection since the line determines the filtration and vice versa. We conclude by Theorem \ref{full bijection between lines and deformations}.
\end{proof}
\begin{Def*}
We call an $n$-$\mathcal{W}_{S}$-display structure $\mathscr{D}$ on $H_{\cris}^{n}(X/W(S))$ \emph{geometric} if $\mathscr{D}=\mathscr{D}_{S}(X')$ for some smooth deformation $X'/S$ of $X$. 
\end{Def*}
\begin{Theorem}\label{geometric displays and CY-type liftings}
There is a bijection
\begin{center}
\begin{tikzpicture}
 \node (a) [align=center] {$ \ \text{ isomorphism classes of } \ $ \\ $\text{ geometric displays on }$ \\ $H_{\cris}^{3}(X/W(S))$};
\node (b)[right =0.8cm of a] {$\{\text{CY-type liftings of }\mathrm{Fil}^{\bullet}\mathscr{D}_{S/R}(X)\}$};

 \draw[->](a) to node [above]{\emph{1:1}}(b);
 \draw [decorate,decoration={brace,amplitude=3pt},xshift=0pt,yshift=0pt]
(-1.8,-0.55) -- (-1.8,0.55)node [black,xshift=0pt] {};
 \draw [decorate,decoration={brace,amplitude=3pt},xshift=0pt,yshift=0pt]
(1.9,0.55) -- (1.9,-0.55)node [black,xshift=0pt] {};
 \end{tikzpicture}
\end{center}  
\end{Theorem}
\begin{proof}
Let $\mathscr{D}$ be a geometric $\mathcal{W}_{S}$-display, i.e. $\mathscr{D}=\mathscr{D}_{S}(Y)$ for some smooth deformation $Y$ of $X$. Write 
\begin{equation*}
\tilde{u}_{\bullet}:\left\{\mathcal{W}_{S}\text{-displays}\right\}\rightarrow \left\{\text{extended }\mathcal{W}_{S/R}\text{-displays}\right\}
\end{equation*}
for the functor induced by $u:\mathcal{W}_{S}\rightarrow\mathcal{W}_{S/R}$, and denote by $\mathscr{D}_{S/R}^{\text{ext}}(X)$ the extended $\mathcal{W}_{S/R}$-display associated to $\mathscr{D}_{S/R}(X)$ (see (\cite{LZ15}, p.465) for the notion of extended display). Then:-
\begin{align*}
\text{Hom}_{\mathcal{W}_{S}}(\mathscr{D},\tilde{u}^{\bullet}\mathscr{D}_{S/R}^{\text{ext}}(X))
& =\text{Hom}_{\mathcal{W}_{S}}(\mathscr{D}_{S}(Y),\tilde{u}^{\bullet}\mathscr{D}_{S/R}^{\text{ext}}(X))\\
& \cong\text{Hom}_{\mathcal{W}_{S/R}}(\tilde{u}_{\bullet}\mathscr{D}_{S}(Y),\mathscr{D}_{S/R}^{\text{ext}}(X))\\
& =\text{Hom}_{\mathcal{W}_{S/R}}(\mathscr{D}_{S/R}^{\text{ext}}(X),\mathscr{D}_{S/R}^{\text{ext}}(X))
\end{align*}
where $\tilde{u}^{\bullet}\mathscr{D}_{S/R}^{\text{ext}}(X)$ is just $\mathscr{D}_{S/R}^{\text{ext}}(X)$ viewed as a $\mathcal{W}_{S}$-predisplay.
\par
The canonical map
\begin{equation*}
\mathscr{D}\rightarrow\tilde{u}^{\bullet}\mathscr{D}_{S/R}^{\text{ext}}(X)=\tilde{u}^{\bullet}\tilde{u}_{\bullet}\mathscr{D}
\end{equation*}
coming from the bijection gives $P_{i}\hookrightarrow\tilde{P}_{i}$, and hence a lifting of $\text{Fil}^{\bullet}\mathscr{D}_{S/R}^{\text{ext}}(X)$, i.e. an admissible lifting of $\text{Fil}^{\bullet}\mathscr{D}_{S/R}(X)$. Moreover, this lifting came from the deformation $Y$, so one sees that it is of CY-type. This gives a map
\begin{equation*}
\resizebox{1.0\hsize}{!}{$
\left\{\text{geometric }\mathcal{W}_{S}\text{-displays on }H_{\cris}^{3}(X/W(S))\right\}\rightarrow\left\{\text{CY-type liftings of Fil}^{\bullet}\mathscr{D}_{S/R}(X)\right\}
$}
\end{equation*}
which only depends on the isomorphism class of the display, so we get an induced map
\begin{center}
\begin{tikzpicture}
 \node (a) [align=center] {$ \ \text{ isomorphism classes of } \ $ \\ $\text{ geometric displays on }$ \\ $H_{\cris}^{3}(X/W(S))$};
\node (b)[right =0.8cm of a] {$\{\text{CY-type liftings of }\mathrm{Fil}^{\bullet}\mathscr{D}_{S/R}(X)\}$};

 \draw[->](a) to node [above]{$\xi$}(b);
 \draw [decorate,decoration={brace,amplitude=3pt},xshift=0pt,yshift=0pt]
(-1.8,-0.55) -- (-1.8,0.55)node [black,xshift=0pt] {};
 \draw [decorate,decoration={brace,amplitude=3pt},xshift=0pt,yshift=0pt]
(1.9,0.55) -- (1.9,-0.55)node [black,xshift=0pt] {};
 \end{tikzpicture}
\end{center}  
In order to prove that $\xi$ is a bijection, we shall construct its inverse:-
\par 
By (\cite{LZ15}, Corollary 11) there is an equivalence of categories 
\begin{center}
\begin{tikzpicture}
 \node (a) [align=center] {$\mathcal{W}_{S}\text{-displays } \ $};
\node (b) [align=center] [right =0.8cm of a]{$\mathcal{W}_{S/R}\text{-displays together with an}$ \\ $ \ \text{ admissible lifting of the Hodge filtration}$};
\node (z) [below =0.4cm of a]{};
\node (c) [right =0.1cm of z]{$\mathscr{D}$};
\node (d) [right =2.9cm of c]{$(u_{\bullet}\mathscr{D},\text{Fil}^{\bullet}\mathscr{D}$)};
\node (e) [below =0.4cm of c]{$\tilde{\mathscr{D}}_{E^{\bullet}}$};
\node (f) [right =2.8cm of e]{$(\tilde{\mathscr{D}},E^{\bullet})$};
 \draw[->](a) to node [above]{$\sim$}(b);
 \draw[|->](c) to node {}(d);
 \draw[|->](f) to node {}(e);
 \draw [decorate,decoration={brace,amplitude=2pt},xshift=0pt,yshift=0pt]
(-1.1,-0.25) -- (-1.1,0.25)node [black,xshift=0pt] {};
 \draw [decorate,decoration={brace,amplitude=2pt},xshift=0pt,yshift=0pt]
(1,0.25) -- (1,-0.25)node [black,xshift=0pt] {};
 \draw [decorate,decoration={brace,amplitude=2pt},xshift=0pt,yshift=0pt]
(2.2,-0.5) -- (2.2,0.5)node [black,xshift=0pt] {};
\draw [decorate,decoration={brace,amplitude=2pt},xshift=0pt,yshift=0pt]
(8.6,0.5) -- (8.6,-0.5)node [black,xshift=0pt] {};
\end{tikzpicture}
\end{center} 
Since CY-type liftings of $\text{Fil}^{\bullet}\mathscr{D}_{S/R}(X)$ are admissible, the above functor induces a map
\begin{center}
\begin{tikzpicture}
\node (a) {$\{\text{CY-type liftings of }\mathrm{Fil}^{\bullet}\mathscr{D}_{S/R}(X)\}$}; 
 \node (b) [right =0.8cm of a][align=center] {$ \ \text{ isomorphism classes } \ $ \\ \text{of }$\mathcal{W}_{S}\text{-displays on }$ \\ $H_{\cris}^{3}(X/W(S))$};
 \node (c) [below =0.8cm of a]{$E^{\bullet}$};
\node (d) [right =4cm of c] {$[\mathscr{D}_{S/R}(X)_{E^{\bullet}}]$};
 \draw[->](a) to (b);
 \draw [decorate,decoration={brace,amplitude=3pt},xshift=0pt,yshift=0pt]
(3.9,-0.55) -- (3.9,0.55)node [black,xshift=0pt] {};
 \draw [decorate,decoration={brace,amplitude=3pt},xshift=0pt,yshift=0pt]
(7.1,0.55) -- (7.1,-0.55)node [black,xshift=0pt] {};
\draw [|->](c) to (d);
 \end{tikzpicture}
\end{center}  
First we show that the image of this map is geometric. Recall from Theorem \ref{full bijection between lines and deformations} that there is a bijection 
\begin{center}
\begin{adjustbox}{width=12cm}
\begin{tikzpicture}
 \node (a) [align=center] {$\text{ isomorphism classes of \ }$ \\ $\text{deformations of }X/R \text{ over }S \ \ $};
\node (b) [align=center] [right =0.8cm of a]{$ \ \text{ direct summands }$ \\ \ $E\subset\sum_{i+j=3}\mathfrak{a}^{[i]}\fil^{j}H_{\mathrm{dR}}^{3}(X'/S)$ \\ $\text{ lifting }\text{Fil}^{3}H_{\dR}^{3}(X/R) \text{ such}$ \\ $\text{that }(E\oplus\nabla E)^{\perp}=E\oplus\nabla E$};
\node (z) [below =0.6cm of a]{};
\node (c) [right =0.01cm of z]{$[Y]$};
\node (d) [right =2.9cm of c]{$\beta([Y]):=\Psi^{-1}(\text{Fil}^{3}H_{\dR}^{3}(Y/S))$};
\node (e) [left =0.1cm of a]{$\beta:$};
 \draw[->](a) to node [above]{\small{1:1}}(b);
 \draw[|->](c) to node {}(d);
 \draw [decorate,decoration={brace,amplitude=2pt},xshift=0pt,yshift=0pt]
(-2.45,-0.5) -- (-2.45,0.5)node [black,xshift=0pt] {};
 \draw [decorate,decoration={brace,amplitude=2pt},xshift=0pt,yshift=0pt]
(2.25,0.5) -- (2.25,-0.5)node [black,xshift=0pt] {};
 \draw [decorate,decoration={brace,amplitude=2pt},xshift=0pt,yshift=0pt]
(3.4,-0.85) -- (3.4,0.85)node [black,xshift=0pt] {};
\draw [decorate,decoration={brace,amplitude=2pt},xshift=0pt,yshift=0pt]
(8.5,0.85) -- (8.5,-0.85)node [black,xshift=0pt] {};
\end{tikzpicture}
\end{adjustbox}
\end{center} 
where $\Psi:H_{\dR}^{3}(X'/S)\xrightarrow{\sim} H_{\dR}^{3}(Y/S)$ is the ``parallel transport'' isomorphism coming from the Gauss-Manin connection. The preceding lemma then yields a one-to-one correspondence between isomorphism classes of deformations of $X/R$ to $\text{Spec }S$ and CY-type liftings of $\text{Fil}^{\bullet}\mathscr{D}_{S/R}(X)$. 
\par 
Given a CY-type lifting $E^{\bullet}$ of $\text{Fil}^{\bullet}\mathscr{D}_{S/R}(X)$, corresponding to the isomorphism class $[Y]$, we claim that $\mathscr{D}_{S/R}(X)_{E^{\bullet}}\cong\mathscr{D}_{S}(Y)$. Indeed, by the construction of the bijection $\beta$, we have that $E^{\bullet}=\text{Fil}^{\bullet}H_{\dR}^{3}(Y/S)=\text{Fil}^{\bullet}\mathscr{D}_{S}(Y)$. Then the equivalence of categories gives
\begin{equation*}
\mathscr{D}_{S/R}(X)_{E^{\bullet}}=\left(u_{\bullet}\mathscr{D}_{S}(Y)\right)_{\text{Fil}^{\bullet}\mathscr{D}_{S}(Y)}\cong\mathscr{D}_{S}(Y)
\end{equation*}  
So we have two maps:-
\begin{center}
\begin{tikzpicture}
 \node (a) [align=center] {$\text{CY-type liftings } \ $ \\ $\text{ of Fil}^{\bullet}\mathscr{D}_{S/R}(X)$ \ };
\node (b) [align=center] [right =1cm of a]{$ \ \text{isomorphism classes of}$ \\ $ \ \text{ geometric }\mathcal{W}_{S}\text{-displays }$ \\ $\text{ on }H_{\cris}^{3}(X/W(S))$};
\path[->]
([yshift= 2.5pt]a.east) edge node[above] {$\eta$} ([yshift= 2.5pt]b.west)
([yshift= -2.5pt]b.west) edge node[below] {$\xi$} ([yshift= -2.5pt]a.east);

  \draw [decorate,decoration={brace,amplitude=2pt},xshift=0pt,yshift=0pt]
(-1.45,-0.5) -- (-1.45,0.5)node [black,xshift=0pt] {};
 \draw [decorate,decoration={brace,amplitude=2pt},xshift=0pt,yshift=0pt]
(1.25,0.5) -- (1.25,-0.5)node [black,xshift=0pt] {};
 \draw [decorate,decoration={brace,amplitude=2pt},xshift=0pt,yshift=0pt]
(2.75,-0.55) -- (2.75,0.55)node [black,xshift=0pt] {};
\draw [decorate,decoration={brace,amplitude=2pt},xshift=0pt,yshift=0pt]
(6.45,0.55) -- (6.45,-0.55)node [black,xshift=0pt] {};
\end{tikzpicture}
\end{center} 
That these are inverses follows easily from the equivalence of categories:-
\begin{align*}
\xi\circ\eta(E^{\bullet})
& =\xi([\mathscr{D}_{S/R}(X)_{E^{\bullet}}]) \\
& =\mathrm{Fil}^{\bullet}\mathscr{D}_{S/R}(X)_{E^{\bullet}} \\
& =E^{\bullet}
\end{align*}
\begin{align*}
\eta\circ\xi([\mathscr{D}_{S}(Y)])
& =\eta(\mathrm{Fil}^{\bullet}\mathscr{D}_{S}(Y)) \\
& = [\mathscr{D}_{S/R}(X)_{\mathrm{Fil}^{\bullet}\mathscr{D}_{S}(Y)}] \\
& =[u_{\bullet}\mathscr{D}_{S}(Y)_{\mathrm{Fil}^{\bullet}\mathscr{D}_{S}(Y)}] \\
& = [\mathscr{D}_{S}(Y)]
\end{align*}
\end{proof}
Write $\underline{\mathcal{CY}3}_{S}$ for the category whose objects are deformations of $X_{0}$ over $\mathrm{Spec}\,S$, and whose morphisms are isomorphisms. Write $\underline{\mathcal{DCY}3}_{S/R}$ for the category whose objects are pairs $(X,E^{\bullet})$ where $X\rightarrow\text{Spec }R$ is a deformation of $X_{0}$ and  $E^{\bullet}$ is a CY-type lifting of $\text{Fil}^{\bullet}\mathscr{D}_{S/R}(X)$. The morphisms are isomorphisms with the obvious compatibilities. By Theorem \ref{geometric displays and CY-type liftings}, $\underline{\mathcal{DCY}3}_{S/R}$ can also be understood to be the category of pairs $(X',[\mathscr{D}])$ where $\mathscr{D}$ is a $\mathscr{W}_{S}$-display structure on $H^{3}_{\mathrm{cris}}(X/W(S))$.
\par
Let $X'\in\underline{\mathcal{CY}3}_{S}$ and write $X'_{R}$ for its reduction over $\text{Spec }R$. The proof of Theorem \ref{geometric displays and CY-type liftings} showed that $\text{Fil}^{\bullet}\mathscr{D}_{S}(X')$ is a CY-type lifting of $\text{Fil}^{\bullet}\mathscr{D}_{S/R}(X'_{R})$. By the compatibility of crystalline cohomology with base-change, we get a functor
\begin{equation*}
\underline{\mathcal{CY}3}_{S}\rightarrow\underline{\mathcal{DCY}3}_{S/R} \ \ ; \ \ X'\mapsto (X'_{R},\text{Fil}^{\bullet}\mathscr{D}_{S}(X'))
\end{equation*} 
We can then rephrase the results of this section as the following:-
\begin{Theorem}\label{CY3 and DCY3}
The above functor is full and essentially surjective. 
\end{Theorem}
\begin{proof}
We have shown that the functor is essentially surjective. It is left to show that it is also full. Let $X$ be a deformation of $X_{0}$ over $\mathrm{Spec}\,R$, and let $Y$ be a deformation of $X$ over $\mathrm{Spec}\,S$. Let $\alpha:X\xrightarrow{\sim} X$ be an automorphism of $X$ (just as an $R$-scheme, not necessarily as a deformation). The claim, then, is that $\alpha$ lifts to an automorphism of $Y$ if and only if $\alpha^{\ast}:H^{3}_{\mathrm{cris}}(X/S)\rightarrow H^{3}_{\mathrm{cris}}(X/S)$ respects $\mathrm{Fil}^{\bullet}\mathscr{D}_{S}(Y)=\mathrm{Fil}^{\bullet}H^{3}_{\mathrm{dR}}(Y/S)$. Now, $X$ being a deformation of $X_{0}$ means that $X$ is an $R$-scheme with a given isomorphism $\rho:X_{0}\xrightarrow{\sim}X\times_{\mathrm{Spec}\,R}\mathrm{Spec}\,k$. Let $\alpha_{0}:X_{0}\xrightarrow{\sim}X_{0}$ be the automorphism of $X_{0}$ induced by $\alpha$. Then to give a lifting $\alpha':Y\xrightarrow{\sim}Y$ of $\alpha$ is equivalent to giving an isomorphism $(Y,\rho')\cong (Y,\rho'\circ\alpha_{0})$ of deformations of $X_{0}$. Here $\rho':X_{0}\xrightarrow{\sim}Y\times_{\mathrm{Spec}\,S}\mathrm{Spec}\,k$ is the given isomorphism making $Y$ a deformation of $X_{0}$. We conclude by Theorem \ref{defs and CY-type liftings}.
\end{proof}
\begin{Corollary}\label{CY3 and DCY3 equivalence}
If $H^{0}(X_{0},\mathcal{T}_{X_{0}/k})=0$ then the above functor is an equivalence of categories.
\end{Corollary}
\begin{proof}
The group $H^{0}(X_{0},\mathcal{T}_{X_{0}/k})$ controls the infinitesimal automorphisms. In the notation of the previous proof, $H^{0}(X_{0},\mathcal{T}_{X_{0}/k})=0$ implies that there are no non-trivial automorphisms of $Y$ deforming the identity on $X$. 
\end{proof}
Note that the condition in Corollary \ref{CY3 and DCY3 equivalence} implies that $\mathfrak{S}$ is universal.

\end{document}